\documentclass{amsart}
\usepackage{appendix}
\usepackage{amstext}
\usepackage{amsthm}
\usepackage{amssymb}
\usepackage{enumerate}
\usepackage{color}
\usepackage[bookmarksnumbered]{hyperref}
\usepackage{dsfont}

\newtheorem{theorem}{Theorem}[section]
\newtheorem{proposition}[theorem]{Proposition}
\newtheorem{lemma}[theorem]{Lemma}

\newtheorem{construction}[theorem]{Construction}
\theoremstyle{definition}
\newtheorem{definition}[theorem]{Definition}
\newtheorem{remark}[theorem]{Remark}

\newcommand{\M}{\mathcal{M}}
\newcommand{\e}{\varepsilon}
\newcommand{\tr}{\mbox{\rm tr}}

\begin{document}
\title[Diagonality modulo symmetric spaces in semifinite vNa]{Diagonality modulo symmetric spaces in semifinite von Neumann algebras}

\begin{abstract} In the study on the diagonality of an $n$-tuple $\alpha=(\alpha(j))_{j=1}^n$ of commuting self-adjoint operators modulo a given $n$-tuple $\varPhi=(\mathcal{J}_1,\ldots,\mathcal{J}_n)$ of normed ideals in $B(H)$, Voiculescu introduced the notion of quasicentral modulus $k_{\varPhi}(\alpha)$ and proved that $\alpha$ is diagonal modulo $(\mathcal{J}_1,\ldots,\mathcal{J}_n)$ if and only if $k_{\varPhi}(\alpha)=0.$ We prove that the same assertion holds true when $B(H)$ is replaced with a $\sigma$-finite semifinite von Neumann algebra $\M$, and $\mathcal{J}_1,\ldots,\mathcal{J}_n$ are replaced with symmetric spaces $E_1(\M),\ldots,E_n(\M)$ associated with $\M.$
\end{abstract}

\keywords{Weyl-von Neumann theorem, Voiculescu's theorem, Diagonal operator, Normed ideal, von Neumann algebra}

\author[A.\ Ber]{Aleksey Ber}
\address{V.I.Romanovskiy Institute of Mathematics Uzbekistan Academy of Sciences
    University street, 9, Olmazor district, Tashkent, 100174, Uzbekistan,
    National University of Uzbekistan
    4, Olmazor district, Tashkent, 100174, Uzbekistan}
\email{aber1960@mail.ru}

\author[F.\ Sukochev]{Fedor Sukochev}
\address{School of Mathematics and Statistics, University of New South Wales, Kensington, NSW 2052, Australia}
\email{f.sukochev@unsw.edu.au}

\author[D.\ Zanin]{Dmitriy Zanin}
\address{School of Mathematics and Statistics, University of New South Wales, Kensington, NSW 2052, Australia}
\email{d.zanin@unsw.edu.au}

\author[H.\ Zhao]{Hongyin Zhao}
\address{School of Mathematics and Statistics, University of New South Wales, Kensington, NSW 2052, Australia}
\email{hongyin.zhao@unsw.edu.au}

\maketitle

\section{Introduction}\label{intro sec}

Let $H$ be a (possibly non-separable) Hilbert space and let $B(H)$ be the $\ast$-algebra of all bounded operators on $H$. An operator $d$ in $B(H)$ is called {\em diagonal} if there is a sequence of pairwise orthogonal projections $\{p_k\}_{k\in \mathbb{N}}$ in $B(H)$ such that $d=\sum_{k\in \mathbb{N}}\lambda_k p_k$, where the series converges in the strong operator topology and $\{\lambda_k\}_{k\in \mathbb{N}}$ is a sequence of complex numbers. A commuting $n$-tuple $(\delta(j))_{j=1}^n\in(B(H))^n$ is called a {\em diagonal $n$-tuple} if $\delta(j)$ is diagonal for each $1\le j\le n.$

In what follows, the symbol $H_0$ always refers to a separable Hilbert space. The classical Weyl’s theorem \cite{Weyl1909} states that every self-adjoint operator $b$ on $H_0$ admits the representation $b=d+c$, where $d$ is a diagonal self-adjoint operator and $c$ is a compact operator. Later, in 1935, von Neumann \cite{von1935} improved the result by replacing a “compact operator” with an “arbitrarily small Hilbert-Schmidt operator”. Kuroda \cite{Kuroda1958} showed that one can require $c$ to belong to any normed ideal in $B(H_0)$ other than the trace class. 

Kato \cite{Kato1957, Kato1957_2}, Rosenblum \cite{Rosenblum1957}, and Birman \cite{Birman1962, Birman1963} showed that, for any self-adjoint operators $b,b'$ on $H_0$ such that $b-b'$ is in the trace class, the absolutely continuous parts of $b$ and $b'$ are unitarily equivalent. Thus, if the spectral measure of $b$ is not singular, then there is no diagonal operator $d\in B(H_0)$ such that $b-d$ is in the trace class.

Berg \cite{Berg1971} proved that, for every normal operator $b$ on $H_0$ and every $\e>0,$ there is a diagonal operator $d$ and a compact operator $c$ such that $b=d+c$ and $\|c\|_{B(H_0)}\le \e.$ The diagonality of a normal operator modulo the Hilbert-Schmidt class is first proved by Voiculescu \cite{V1979}. More specifically, for every normal operator $b\in B(H_0)$ and every $\e>0$, there is a diagonal operator $d\in B(H_0)$ such that $\|b-d\|_{L_2(B(H_0))}\le\e$, where $L_2(B(H_0))$ is the Hilbert-Schmidt class.

Voiculescu proved the following more general result \cite{V1979}: if $n\ge2$, then for every commuting $n$-tuple of self-adjoint operators $(\alpha(j))_{j=1}^n\in (B(H_0))^n$ and every $\e>0$, there is a commuting $n$-tuple $(\delta(j))_{j=1}^n\in (B(H_0))^n$ of diagonal operators such that 
$$\|\alpha(j)-\delta(j)\|_{L_n(B(H_0))}\le \e,\quad 1\le j\le n,$$
where $L_n(B(H_0))$ is the Schatten-$n$ class in $B(H_0)$.

Let $(\mathcal{J}_1,\ldots,\mathcal{J}_n)$ be an $n$-tuple of normed ideals (see \cite{V2018}) in $B(H_0)$. We say that an $n$-tuple $\alpha\in (B(H_0))^n$ is {\em diagonal modulo $(\mathcal{J}_1,\ldots,\mathcal{J}_n)$} if there is a diagonal $n$-tuple $\delta\in (B(H_0))^n$ such that $\alpha(j)-\delta(j)\in \mathcal{J}_j$ for $1\le j\le n$ (see \cite{V2018}).

By extending the Kato-Rosenblum theorem to $n$-tuples, Voiculescu \cite{V1979, V2018} proved that a commuting self-adjoint $n$-tuple $\alpha\in (B(H_0))^n$ is diagonal modulo $(L_{n,1}(B(H_0)))^n$ if and only if the spectral measure of $\alpha$ is singular with respect to the Lebesgue measure in $\mathbb{R}^n$. Here $L_{n,1}(B(H_0))$ is the Lorentz-$(n,1)$ ideal in $B(H_0)$ (see \cite{V2018}).

Bercovici and Voiculescu \cite{Bercovici1989} obtained an analogue of Kuroda's theorem for $n$-tuples in $B(H_0)$. More specifically, they proved that if $\mathcal{J}$ is a normed ideal such that $\mathcal{J}\not\subset L_{n,1}(B(H_0))$, where $L_{n,1}(B(H_0))$ is the Lorentz-$(n,1)$ ideal in $B(H_0)$, then for every commuting self-adjoint $n$-tuple $\alpha\in (B(H_0))^n$ and every $\e>0$, there is a diagonal $n$-tuple $\delta\in (B(H_0))^n$ such that $\|\alpha(j)-\delta(j)\|_{\mathcal{J}}\le \e$ for $1\le j\le n$.
 
The key ingredient in Voiculescu's approach \cite{Bercovici1989, V1979, V2018} is the so-called ``quasicentral modulus'' of a given commuting self-adjoint $n$-tuple $\alpha\in (B(H_0))^n$ and a given $n$-tuple $\varPhi=(\mathcal{J}_1,\ldots,\mathcal{J}_n)$ of normed ideals in $B(H_0)$:
\begin{equation}
k_{\varPhi}(\alpha)=\sup_{a\in \mathcal{F}_1^+}\inf_{\stackrel{r\ge a}{r\in \mathcal{F}_1^+}} \max_{1\le j\le n}\|[r,\alpha(j)]\|_{\mathcal{J}_j},
\end{equation}
where $\mathcal{F}_1^+=\{r\in B(H_0):0\le r\le {\bf 1},\ {\rm rank}(r)<\infty\}$. Voiculescu \cite[Corollary 2.6]{V1979} (see also \cite[Proposition 7.1]{V2018}) proved that $\alpha$ is diagonal modulo $(\mathcal{J}_1,\ldots,\mathcal{J}_n)$ if and only if $k_{\varPhi}(\alpha)=0.$ The quasicentral modulus underlies many questions concerning perturbation of operators \cite{V1979, V1981, V1990, V1992,V2018}. We refer the reader to \cite{V2019} for a comprehensive survey about more use of quasicentral modulus.

The main aim of this paper is to deliver results generalizing and extending \cite[Corollary 2.6]{V1979} (see also \cite[Proposition 7.1]{V2018}) for general semifinite von Neumann algebras. A \emph{von Neumann algebra} $\M$ is a $\ast$-subalgebra of $B(H)$ such that $\M=\M''$, where $\M''$ is the  bicommutant of $\M.$  A von Neumann algebra $\M$ is called {\em semifinite} if there exists a faithful normal semifinite trace $\tau$ on $\M$ (see \cite{T1}). A {\em factor} is a von Neumann algebra with trivial centre. A von Neumann algebra $\M$ is \emph{$\sigma$-finite} if each orthogonal family of non-zero projections in $\M$ is countable. 

If $d=\sum_{k}\lambda_kp_k\in\M$ is a diagonal operator, then it can be written as $d=\sum_{k}\lambda_k' p_k'$ where $\{p_k'\}$ is a sequence pairwise orthogonal projections in $\M$ such that $\sum_{k}p_k'={\bf 1}$ and $\{\lambda_k'\}$ is a sequence of pairwise different complex numbers.

Let us briefly recall the definition of symmetric spaces associated to a semifinite von Neumann algebra, which is an analogue of the normed ideals in $B(H_0)$. Suppose $\M$ is a von Neumann algebra with a faithful normal semifinite trace $\tau$. Let $S(\tau)$ be the set of all $\tau$-measurable operators affiliated with $\M$ (see \cite{DPS2023, Fack1986}), and let $\mu(a)$ denote the singular value function of $a\in S(\tau)$ (specific definitions of $S(\tau)$ and $\mu(a)$ will be given below, in Section \ref{pre section}).

Let $E$ be a symmetric function space on $(0,\infty)$ with norm $\|\cdot\|_E$ (see Definition \ref{sym function space def}). Define $E(\M)=\{a\in S(\tau):\mu(a)\in E\},$ and define 
$$\|a\|_{E(\M)}=\|\mu(a)\|_{E},\quad a\in E(\M).$$
From \cite{KS2008} (see also \cite[Theorem 3.5.5]{LSZ2013}), $(E(\M),\|\cdot\|_{E(\M)})$ is a Banach space and is called a {\em symmetric space}. In particular, for $1\le p< \infty$, if $E=L_p$ is the standard Lebesgue $L_p$ function space on $(0,\infty)$, we obtain the classical noncommutative $L_p$ spaces $L_p(\M)$, see e.g. \cite{DPS2023, PX2003}. 
For convenience, we set $L_\infty(\M)=\M$ equipped with the uniform norm $\|\cdot\|_{\M}.$ If $E=L_{p,1}$ is the standard Lorentz-$(p,1)$ space on $(0,\infty)$, we obtain the noncommutative Lorentz space $L_{p,1}(\M)$, see e.g. \cite{DPS2023, Kosaki1981}. 

Let $\M$ be a $\sigma$-finite semifinite factor with the uniform norm $\|\cdot\|_{\M}$ and let $\mathcal{K}(\M)$ be the two-sided closed ideal generated by all finite projections in $\M$. In 1975, Zsido \cite{Zsido1975} showed that, for every self-adjoint operator $b$ in $\M$, there exists a diagonal operator $d$ in $\M$ such that $b-d\in \mathcal{K}(\M)$.
This result was further extended by Akemann and Pedersen \cite{Akemann1977}, where they proved that for $\e>0$, there is a diagonal operator $d\in\M$ such that $\|b-d\|_{\M}\le \e.$ Later, in 1978, Kaftal \cite{Kaftal1978} proved that, for every self-adjoint operator $b$ in $\M$ and $\e>0$, there exists a diagonal operator $d$ in $\M$ such that $\max\{\|b-d\|_{L_2(\M)},\|b-d\|_{\M}\}\le\e$.

Recently, Li et al. \cite{LSS2020} proved that if $\M$ is a semifinite von Neumann algebra acting on a separable Hilbert space with separable predual, then for every normal operator $b\in\M$ and every $\e>0,$ there is a diagonal operator $d\in \M$ such that $\max\{\|b-d\|_{L_2(\M)},\|b-d\|_{\M}\}\le \e$ (see \cite[Theorem 6.2.5]{LSS2020}).

When $\M$ is a $\sigma$-finite semifinite factor, Li et al. \cite[Theorem 6.1.2]{LSS2020} proved that if $n\ge2$, then for every commuting self-adjoint $n$-tuple $\alpha\in\M^n$ and every $\e>0$, there is a diagonal $n$-tuple $\delta\in\M^n$ such that $\|\alpha(j)-\delta(j)\|_{L_n(\M)}\le \e$ and $\|\alpha(j)-\delta(j)\|_{\M}\le \e$ for each $1\le j\le n.$

In this paper, we let $\M$ be a von Neumann algebra with a faithful normal semifinite trace $\tau$, acting on a (possibly non-separable) Hilbert space $H$, and consider the problem of the diagonality of every commuting self-adjoint $n$-tuple $(\alpha(j))_{j=1}^n\in\M^n$ modulo a given $n$-tuple of symmetric spaces associated with $\M$.

Let $\varPhi=(E_1,\ldots,E_n)$ be an $n$-tuple of symmetric function spaces on $(0,\infty)$. Denote
$$\varPhi(\M)=(E_1(\M),\ldots,E_n(\M)).$$
For every $n$-tuple $\beta=(\beta(j))_{j=1}^n\in \varPhi(\M),$ we write
$$\|\beta\|_{\varPhi(\M)}=\max_{1\le j\le n} \|\beta(j)\|_{E_j(\M)}.$$
In particular, if $\varPhi=E^n$, we also write $\|\beta\|_{\varPhi(\M)}$ as $\|\beta\|_{E(\M)}$.

Let $\alpha\in\M^n$ be a commuting $n$-tuple, we say that $\alpha$ is {\em diagonal modulo $\varPhi(\M)$} if there is a diagonal $n$-tuple $\delta\in \M^n$ such that $\alpha(j)-\delta(j)\in E_j(\M)$ for $1\le j\le n.$

We now recall the definition of quasicentral modulus $k_{\varPhi(\M)}(\alpha)$ in the semifinite von Neumann algebra setting (see \cite{BSZZ2023, V2021, V2022}). General properties of $k_{\varPhi(\M)}(\alpha)$ can be found in \cite{BSZZ2023}. For every $n$-tuple $\alpha\in \M^n$ and every $a\in\M,$ we write $[a,\alpha]=([a,\alpha(j)])_{j=1}^n.$
\begin{definition}\label{quasicentral modulus def}
Let $\varPhi=(E_1,\ldots,E_j)$ be an $n$-tuple of symmetric function spaces on $(0,\infty)$. For every $n$-tuple $\alpha\in \M^n$, we define the \emph{quasicentral modulus} $k_{\varPhi(\M)}(\alpha)$ of $\alpha$ by the formula (see Definition A.1 in \cite{BSZZ2023})
$$k_{\varPhi(\M)}(\alpha)= \sup_{a\in \mathcal{F}_1^+(\M)}\inf_{\stackrel{r\ge a}{r\in \mathcal{F}_1^+(\M)}}\|[r,\alpha]\|_{\varPhi(\M)}.$$
Here $\mathcal{F}_1^+(\M)=\{x\in\M:0\le x\le \mathbf{1},\tau(\mathfrak{l}(x))<\infty\}$, where $\mathfrak{l}(x)$ is the left support of $x$ (see Section \ref{pre section}). In particular, if $\varPhi=E^n$, we also write $k_{\varPhi(\M)}(\alpha)$ as $k_{E(\M)}(\alpha)$.
\end{definition}

This paper is a continuation of our previous investigations \cite{BSZZ2023} on the diagonality of commuting tuples of self-adjoint operators in a semifinite von Neumann algebra $\M$ modulo symmetric spaces associated with $\M$. Suppose $\M$ is an infinite $\sigma$-finite semifinite factor. We proved in \cite[Theorem 1.2]{BSZZ2023} that, $\alpha$ is diagonal modulo $\varPhi(\M)$ if and only if $k_{\varPhi(\M)}(\alpha)=0$. In this paper, we prove that the same assertion holds true for arbitrary $\sigma$-finite semifinite von Neumann algebras. Our techniques in this paper are quite different from that of \cite{BSZZ2023}.

The following theorem is the main result of this paper. It is an analogue of \cite[Corollary 2.6]{V1979}, and it extends \cite[Theorem 1.2]{BSZZ2023} to non-factors. Let $\mathcal{K}(\M,\tau)$ denote the two-sided closed ideal
in $\M$ generated by all $\tau$-finite projections. We write $\varPhi^0=(E_1^0,\ldots,E_n^0)$, where $E_j^0=\overline{L_1\cap L_\infty}^{\|\cdot\|_{E_j}}$ for each $1\le j\le n.$ We denote $\varPhi^0(\M)=(E_1^0(\M),\ldots,E_n^0(\M))$.

\begin{theorem}\label{main theorem} Let $\M$ be a $\sigma$-finite von Neumann algebra with a faithful normal semifinite trace $\tau$ and let $n\in \mathbb{N}$. Let $\varPhi=(E_1,\ldots,E_n)$ be an $n$-tuple of symmetric function spaces on $(0,\infty).$ Suppose $\alpha\in\M^n$ is a commuting self-adjoint $n$-tuple. The following statements are equivalent:
\begin{enumerate}[\rm (i)]
\item\label{mta} $k_{\varPhi(\M)}(\alpha)=0$.
\item\label{mtb} For every $\e>0,$ there exists a diagonal $n$-tuple $\delta\in \M^n$ such that $\alpha-\delta\in
\varPhi^0(\M)$ and $\|\alpha -\delta\|_{\varPhi(\M)}\le \e.$
\end{enumerate}
\end{theorem}

\begin{remark} If $k_{\varPhi(\M)}(\alpha)=0$, then $k_{(\varPhi\cap L_\infty)(\M)}(\alpha)=0$ (see Lemma \ref{quasicentral modulus on bounded part}), where $\varPhi\cap L_\infty=(E_1\cap L_\infty,\ldots,E_n\cap L_\infty)$. Thus, by Theorem \ref{main theorem}, for every $\e>0,$ there is a diagonal $n$-tuple $\delta\in\M^n$ such that $\alpha-\delta\in \varPhi^0(\M)$ and $$\max\{\|\alpha-\delta\|_{\varPhi(\M)},\|\alpha-\delta\|_{\M}\}\le \e.$$
\end{remark}

In the case when $\M=B(H_0)$ where $H_0$ is a separable Hilbert space, Theorem \ref{main theorem} recovers \cite[Corollary 2.6]{V1979}  and the classical Weyl-von Neumann's theorem \cite{von1935, Weyl1909}.

A combination of Theorem \ref{main theorem} with \cite[Lemma 4.3.3]{LSS2020} yields that if $n\ge2$, then for every commuting self-adjoint $n$-tuple $\alpha\in\M^n$ and every $\e>0$, there is a diagonal $n$-tuple $\delta\in\M^n$, such that $\|\alpha(j)-\delta(j)\|_{L_n(\M)}\le \e$ and $\|\alpha(j)-\delta(j)\|_{\M}\le \e$ for $1\le j\le n$. This extends \cite[Theorem 6.1.2]{LSS2020} to non-factors.

Note that for any normal operator $b\in \M$, $\alpha=({\rm Re}(b),{\rm Im}(b))$ is a commuting self-adjoint $2$-tuple. Then the above paragraph implies that for every $\e>0,$ there is a diagonal operator $d\in \M$ such that $\max\{\|b-d\|_{L_2(\M)},\|b-d\|_{\M}\}\le\e$. This extends \cite[Theorem 6.2.5]{LSS2020} to von Neumann algebras acting on a non-separable Hilbert space with non-separable predual. Our techniques are quite different to that of \cite{LSS2020}.

%

The proof of Theorem \ref{main theorem} is based on an analogue (see Theorem \ref{key thm}) of Voiculescu's approximate homomorphisms theorem \cite[Corollary 2.5]{V1979}, which is of independent interest. Before stating this result, let us recall the notion of approximate equivalence of two representations $\psi,\rho$ of a $C^\ast$-algebra $\mathcal{A}$ when the ranges of $\psi$ and $\rho$ are contained in a von Neumann algebra $\M$.

Suppose $\mathcal{A}$ is a unital $C^\ast$-algebra and $\M$ is a von Neumann algebra. Let $\psi,\rho:\mathcal{A}\to \M$ be two unital $\ast$-homomorphisms. We say that $\psi$ and $\rho$ are {\em approximately equivalent in $\M$} if there is a sequence of unitaries $\{u_m\}_{m\in \mathbb{N}}$ in $\M$ such that 
$$\psi(a)=\lim_{m\to\infty} u_m^{-1}\rho(a)u_m,\quad a\in \mathcal{A},$$
and we write $\psi\sim_{\M} \rho.$ In the special case when $\M=B(H)$ and when $\mathcal{A}$ and $H$ are both separable, Voiculescu \cite{V1976} gave a beautiful characterization of approximate equivalence (see also \cite{Arveson1977} for an exposition of Voiculescu's theorem). Hadwin \cite{Hadwin1981} (see also \cite{Davidson1996}) proved that Voiculescu's characterization could be formulated in terms of the rank function; more precisely, $\psi\sim_{B(H)}\rho$ if and only if ${\rm rank}(\psi(a))={\rm rank}(\rho(a))$ for all $a\in \mathcal{A}$. Hadwin \cite{Hadwin1981} also proved that the rank characterization holds when $\mathcal{A}$ or $H$ are non-separable. For a general von Neumann algebra $\M$ acting on a separable Hilbert space and for a unital commutative $C^\ast$-algebra $\mathcal{A}$, Ding and Hadwin \cite[Theorem 3]{Ding2005} proved that $\psi\sim_{\M}\rho$ if and only if $\mathfrak{l}(\psi(a))\sim \mathfrak{l}(\rho(a))$ for all $a\in \mathcal{A}$, where $\mathfrak{l}(x)$ represents the left support projection (see Section \ref{pre section}) of $x\in\M$.

Let $\mathcal{Z}(\M)$ denote the centre of $\M.$ For a subset $A\subset \M$, let $C^\ast(A)$ (respectively, $W^\ast(A)$) denote the $C^\ast$-subalgebra (respectively, von Neumann subalgebra) generated by $A$ and the identity $\mathbf{1}$. Denote by $WZ^\ast(A)$ the von Neumann subalgebra of $\M$ generated by $W^\ast(A)\cup \mathcal{Z}(\M)$. Let $\mathcal{P}_f(\M)$ denote the set of all finite projections in $\M$ (see Section \ref{pre section}).

The following result shows that, for every $\ast$-monomorphism $\psi:C^\ast(\alpha)\to\M$ which satisfies $\psi\sim_{\M}{\rm Id}_{C^{\ast}(\alpha)},$ $\psi(\alpha)$ can be unitarily approximated by $\alpha$ in the sense of the norm $\|\cdot\|_{\varPhi(\M)}$ where $\varPhi$ is a given $n$-tuple of symmetric function spaces on $(0,\infty)$. This is an analogue of Corollary 2.5 in \cite{V1979}.

\begin{theorem}\label{key thm} Let $\M$ be a properly infinite $\sigma$-finite von Neumann algebra with a faithful normal semifinite trace $\tau$. Suppose $\varPhi=(E_1,\ldots,E_n)$ is an $n$-tuple of symmetric function spaces on $(0,\infty)$. Let $\alpha\in\M^n$ be a commuting self-adjoint $n$-tuple and let $\psi:C^{\ast}(\alpha)\to\M$ be a $\ast$-monomorphism. Suppose that
\begin{enumerate}[\rm (i)]
\item\label{assum1} $WZ^{\ast}(\alpha)\cap \mathcal{P}_f(\M)=WZ^{\ast}(\psi(\alpha))\cap \mathcal{P}_f(\M)=\{0\};$
\item\label{assum2} $\psi\sim_{\M} {\rm Id}_{C^{\ast}(\alpha)};$
\item\label{assum3} $k_{\varPhi(\M)}(\alpha)=k_{\varPhi(\M)}(\psi(\alpha))=0.$ 
\end{enumerate}
For every $\e>0$ there exists a unitary $u\in\M$ such that $u\psi(\alpha)-\alpha u\in \varPhi^0(\M)$ and
$$\|u\psi(\alpha)-\alpha u\|_{\varPhi(\M)}\leq\varepsilon.$$
\end{theorem}

In the special case when $\M=B(H_0)$, Theorem \ref{key thm} is a slightly weaker result comparing to \cite[Corollary 2.5]{V1979} since the condition \eqref{assum1} amounts to say that $W^\ast(\alpha)\cap \mathcal{K}=W^\ast(\psi(\alpha))\cap \mathcal{K}=\{0\}$, which is stronger than the original condition that $C^\ast(\alpha)\cap \mathcal{K}=C^\ast(\psi(\alpha))\cap \mathcal{K}=\{0\}$ in \cite[Corollary 2.5]{V1979}, where $\mathcal{K}$ is the ideal of compact operators in $B(H_0)$. In this case, the condition \eqref{assum2} is redundant. Indeed, by \cite[Corollary 1.4]{V1976}, the assumption that $\psi$ is injective and the condition \eqref{assum1} imply automatically that $\psi\sim_{B(H_0)}{\rm Id}_{C^{\ast}(\alpha)}$.

Section \ref{pre section} contains some preliminaries. Section \ref{qm sec} contains some useful properties of quasicentral modulus. In Section \ref{construction section} we construct a $\ast$-monomorphism $\psi:C^\ast(\alpha)\to\M$ such that $\psi\sim_{\M}{\rm Id}_{C^{\ast}(\alpha)}$ and such that $\psi(\alpha)$ can be approximated by diagonal $n$-tuples in the sense of any norm $\|\cdot\|_{\varPhi(\M)}$ where $\varPhi$ is an $n$-tuple of symmetric function spaces on $(0,\infty)$. Section \ref{technical section} contains a technical result for the proof of Theorem \ref{key thm}. Section \ref{V homomorphism section} contains the proof of Theorem \ref{key thm}. Finally, using Theorem \ref{key thm} and the construction in Section \ref{construction section}, we prove Theorem \ref{main theorem} in Section \ref{final section}.

\section{Preliminaries}\label{pre section}

In this section, we recall some notions in the von Neumann algebra theory, and define the symmetric spaces associated with a semifinite von Neumann algebra.

In what follows, $H$ is a Hilbert space and $B(H)$ is the $\ast$-algebra of all bounded linear operators on $H$, and $\mathbf{1}$ is the identity operator on $H$. Suppose $\mathcal{A}$ is a unital $\ast$-algebra on $H$. Recall that $\mathcal{A}$ is called a von Neumann algebra if $\mathcal{A}=\mathcal{A}''$ where $\mathcal{A}''$ is the  bicommutant of $\mathcal{A}$. From the famous von Neumann bicommutant theorem (see e.g. \cite[Theorem 3.2]{Stratila2019}), the closure of $\mathcal{A}$ in the strong operator topology (briefly, $so$-topology), denoted by $\overline{\mathcal{A}}^{so}$, coincides with the bicommutant $\mathcal{A}''$ of $\mathcal{A}$. This algebra is the von Neumann algebra generated by $\mathcal{A}$. 

For an operator $x\in B(H)$, we denote by $\mathfrak{l}(x)$ the projection onto the subspace $\overline{x(H)}$, which is called the \textit{left support} of $x$, i.e.\ $\mathfrak{l}(x)$ is the smallest projection $e\in B(H)$ for which $ex=x$. We denote by $\mathfrak{r}(x)$ the projection onto the subspace $(\ker x)^\perp$, which is called the \textit{right support} of $x$, i.e.\ $\mathfrak{r}(x)$ is the smallest projection $e\in B(H)$ for which $xe=x$. 

Let $\M$ be a von Neumann algebra acting on $H$. The centre of $\M$ is denoted by $\mathcal{Z}(\M)$, namely, $\mathcal{Z}(\M)=\M\cap \M'.$ We let $\mathcal{P}(\M)$ denote the set of all projections in $\M.$ For two projections $p,q\in \mathcal{P}(\M)$, we say that $p\sim q$ if there is a partial isometry $v\in\M$ such that $v^\ast v=p$ and that $vv^\ast=q$ (see \cite[Section V.1]{T1}).  A projection $p\in \mathcal{P}(\M)$ is said to be {\em finite} if $p\sim q\le p$ implies $p=q.$ Otherwise, it is said to be {\em infinite}. Let $\mathcal{P}_f(\M)$ denote the set of all finite projections in $\M$. A projection $p\in \mathcal{P}(\M)$ is called {\em properly infinite} if $zp$ is infinite for every central projection $z\in \M$ with $zp\ne0$ (see \cite[Definition V.1.15]{T1}). A von Neumann algebra is called {\em finite} (respectively, {\em properly infinite}) if the identity operator $\mathbf{1}$ is finite (respectively, properly infinite).

Suppose $\{x_i\}_{i\in I}$ is an increasing net of self-adjoint operators in $\M$ and suppose $x\in \M.$ We write $x_i\uparrow x$ if $x$ is the least upper bound of $\{x_i\}_{i\in I}$ in $\M.$

We let $c(a)$ denote the {\em central support} of $a\in\M$, namely, 
$$c(a)=\wedge \{z:z\in \mathcal{P}(\mathcal{Z}(\M)),\ za=a\}.$$

\begin{lemma}\label{equivalent projections lem}
Let $\M$ be a $\sigma$-finite von Neumann algebra and let $p\in \M$ be a properly infinite projection. If $r<p$ is a finite projection in $\M$, then $p-r\sim p.$
\end{lemma}
\begin{proof} For every central projection $z\in \M,$ if $z(p-r)\ne0$ then obviously $zp\ne0.$ Thus, $zp$ is infinite by the definition of properly infinite projection. Since  $zr$ is finite, it follows that $z(p-r)$ is infinite. Hence, $p-r$ is properly infinite.

If $c(p-r)<c(p),$ then $z:=c(p)-c(p-r)\ne0$. We have 
$$z(p-r)=z\cdot c(p-r)(p-r)=0.$$
Clearly, $z\leq c(p)$. Thus, $zp\ne0$ (otherwise $z=zc(p)=c(zp)=0$, which is a contradiction). Now, $0\ne zp=zr$. Hence, $zp$ is non-zero and is finite, which is impossible since $p$ is properly infinite. Therefore, $c(p-r)=c(p).$

Note that $\M$ is $\sigma$-finite. By \cite[Corollary 6.3.5]{Kadison1997_2} we infer that $p-r\sim p.$
\end{proof}

\begin{lemma}\label{finite and tau finite} Let $\M$ be a von Neumann algebra with a faithful normal semifinite trace $\tau$. If $f$ is a non-zero finite projection in $\M$, then 
\begin{enumerate}[\rm (i)]
\item\label{ftfa} there exists $z\in \mathcal{P}(\mathcal{Z}(\M))$ such that $0\ne zf$ is $\tau$-finite;
\item\label{ftfb} there exists central partition of the identity $\{z_i\}_{i\in I}$ in $\M$ such that $z_if$ is $\tau$-finite for each $i\in I$.
\end{enumerate}
\end{lemma}
\begin{proof}  (i) Since the trace $\tau$ is semifinite, there exists a non-zero $\tau$-finite projection $q\in \M$, such that $q\le f$.

By Zorn's lemma there exists a maximal family $\Theta=\{q_i\}_{i\in I}$ of pairwise orthogonal projections in $\M$ such that $q_i\sim q$ and $q_i\leq f$ for every $i\in I.$ 

We claim that $I$ is a finite set. To see the claim, we assume on the contrary that $I$ is infinite. Let $I_1,I_2\subset I$ be disjoint countably infinite subsets. Set
$$e_1=\vee_{i\in I_1}q_i,\quad e_2=\vee_{i\in I_2}q_i,\quad e_3=\vee_{i\in I_1\cup I_2}q_i.$$
Clearly, $e_1,e_2< e_3,$ $e_1\sim e_3,$ and $e_2\sim e_3.$ Thus, $e_3$ is infinite. Since $e_3\leq f,$ it follows that $f$ is infinite. This contradiction proves the claim.

Let
$$e=\sum_{i\in I}q_i.$$
Since $I$ is a finite set and since each $q_i$ is $\tau$-finite, it follows that $e$ is $\tau$-finite too.

By \cite[Theorem 6.2.7]{Kadison1997_2} there exists a projection $z\in \mathcal{P}(\mathcal{Z}(\M))$, such that
\begin{equation}\label{projections comparing}
z(f-e)\precsim zq,\ \text{ and }(\mathbf{1}-z)q\precsim (\mathbf{1}-z)(f-e).
\end{equation}

Suppose that $zf=0.$ Hence, $z(f-e)=0$ and $zq=zqf=qzf=0.$ Therefore, by the second inequality in \eqref{projections comparing}, $q\precsim f-e.$ That is, there exists $q_0\sim q$ such that $q_0\le f-e.$ This contradicts the maximality of $\Theta.$ Thus, $zf\neq 0.$

Since $zq$ is $\tau$-finite, it follows from the first inequality in \eqref{projections comparing} that $z(f-e)$ is $\tau$-finite. Since $e$ is $\tau$-finite, it follows that $ze$ is $\tau$-finite. Thus, $zf$ is $\tau$-finite.

(ii) By Zorn's lemma and part (i) there exists a maximal family $\Gamma=\{z_i\}_{i\in I}$ of pairwise orthogonal central projections such that $0\le\tau(z_if)<\infty$.

Let $z:=\sum_{i\in I}z_i$. Suppose $z<\mathbf{1}$. If $(\mathbf{1}-z)f=0$, we set $z'=(\mathbf{1}-z)$. If $(\mathbf{1}-z)f\ne0$, then part (i) implies that there exists $z'\in \mathcal{P}(\mathcal{Z}(\M))$ such that
 $0<\tau(z'({\bf 1}-z)f)<\infty$. In both cases, $z'':=z'(\mathbf{1}-z)\le \mathbf{1}-z$ is non-zero and $0\le\tau(z''f)<\infty$. However, this contradicts the maximality of $\Gamma$. Therefore, $z=\mathbf{1}$. Thus, $\Gamma$ is the required central partition of the identity.
\end{proof}

The following lemma is simple, and thus we omit the proof.

\begin{lemma}\label{left support lemma} Let $\{p_k\}_{k=1}^m$ be a family of pairwise orthogonal projections in $B(H)$ and let $q\in B(H)$ be a projection. We have 
$$\mathfrak{l}((\sum_{1\le k\le m}p_k) q)\leq\bigvee_{1\le k\le m}\mathfrak{l}(p_kq).$$
\end{lemma}

\subsection{$\tau$-measurable operators}\label{measurable subsec} Let $H$ be a Hilbert space. Let $\M$ be a von Neumann algebra acting on $H$ with a faithful normal semifinite trace $\tau$ and the uniform norm $\|\cdot\|_{\M}$. A linear operator $a:{\rm dom}(a)\to H$ is said to be {\em affiliated with $\M$} if $ba\subset ab$ for all $b$ from the commutant $\M'$ of $\M,$ and the collection of all operators affiliated with $\M$ is denoted by ${\rm Aff}(\M).$ Let $\mathcal{F}(\M,\tau)=\{x\in\M:\tau(\mathfrak{l}(x))<\infty\}.$ Let $\mathcal{P}_f(\M,\tau)$ denote the set of all $\tau$-finite projections in $\M.$ Every $\tau$-finite projection is finite, i.e., $\mathcal{P}_f(\M,\tau)\subset \mathcal{P}_f(\M)$ (see e.g. \cite{T1}). Recall that $\mathcal{K}(\M,\tau)$ is the two-sided closed ideal in $\M$ generated by all $\tau$-finite projections, i.e., $\mathcal{K}(\M,\tau)=\overline{\mathcal{F}(\M,\tau)}^{\|\cdot\|_{\M}}$ (see \cite{Kaftal1977}).

Let $x$ be a self-adjoint operator on $H$. The spectral measure of $x$ is denoted by $e^x:\mathfrak{B}(\mathbb{R})\to \mathcal{P}(\M),$ where $\mathfrak{B}(\mathbb{R})$ represents the collection of all Borel sets in $\mathbb{R}.$ 

A closed densely defined operator $x:{\rm dom}(x)\to H$ affiliated with~$\mathcal{M}$ is said to be \emph{$\tau$-measurable} if there exists $s>0$ such that $\tau(e^{|x|}(s,\infty))<\infty.$
The set of all $\tau$-measurable operators is denoted by $S(\tau)$. It is a $\ast$-algebra w.r.t. the strong sum, strong multiplication (denoted simply by $x+y$ and $xy$, respectively, for $x,y\in S(\tau)$). General facts about $\tau$-measurable operators can be found in \cite{Nelson1974}, \cite{Segal1953}, \cite{Terp1981}, \cite[Section IX.2]{T2} or the book \cite{DPS2023}. 

For $x\in S(\tau)$, the \emph{distribution function} of $x$ is defined by
$$d(s;x)=\tau(e^{|x|}(s,\infty)),\quad s\geq 0.$$

Define 
$$S_0(\tau)=\{x\in S(\tau):d(s;x)<\infty\ \text{ for all }s>0\}.$$
From \cite[Theorem 1.3]{Kaftal1977}, we have
$$S_0(\tau)\cap \M=\mathcal{K}(\M,\tau).$$

\begin{definition}[p. 129 of \cite{DPS2023}] Let $x\in S(\tau).$ The \emph{singular value function} 
$\mu(x):t\mapsto \mu(t;x)$ of the operator $x$, is defined by
$$\mu(t;x)=\inf\{s\geq0: d(s;x)\leq t\},\quad t\geq 0.$$
The function $t\mapsto \mu(t;x)$ is also written as $\mu(x).$ 
\end{definition}

For $x\in S(\tau)$, we have (see e.g. \cite[Proposition 3.2.5]{DPS2023})
$$\mu(t;x)=\inf\{\|xp\|_{\M}:p\in \mathcal{P}(\M),\ p(H)\subset {\rm dom}(x),\ \tau(\mathbf{1}-p)\le t\},\quad t\ge0.$$

We refer the reader to \cite[Chapter 3]{DPS2023} or \cite{Fack1986} for the general properties of singular value functions.

\subsection{Symmetric spaces}\label{symm subsec} We now recall the definition of symmetric spaces associated with a semifinite von Neumann algebra, which is an analogue of the normed ideals in $B(H_0)$. For these notions we follow \cite{DPS2023, KS2008, Krein1982, LSZ2013}.

\begin{definition}\label{sym function space def}
A \emph{symmetric function space} $E$ is a Banach space of
real-valued Borel measurable functions on $(0,\infty)$, equipped with a norm $\|\cdot\|_E$, such that the following condition holds: If $y\in E, x$ is a measurable function and $\mu(x)\leq\mu(y),$ then $x\in E$ and $\|x\|_E\leq \|y\|_E.$
\end{definition}

Recall that $L_p,1\le p\le \infty$, i.e., the Lebesgue $L_p$-spaces on $(0,\infty)$, are symmetric function spaces.

Let $H$ be a Hilbert space. Let $\M$ be a von Neumann algebra acting on $H$ with a faithful normal semifinite trace $\tau$ (see \cite{DPS2023, T1}). Suppose $E$ is a symmetric function space on $(0,\infty)$. Define operator space 
\begin{equation}\label{symetric space from function}
E(\M)=\{x \in S(\tau):\ \mu(x)\in E\},
\end{equation}
and set
\begin{equation}\label{E norm def}
\left\|x \right\|_{E(\mathcal{M})}=\left\|\mu(x )\right\|_E,\quad x\in E(\M).
\end{equation}
From \cite{KS2008} (see also its exposition in \cite[Theorem 3.5.5]{LSZ2013}), $(E(\M),\|\cdot\|_{E(\M)})$ is a Banach space and is called a {\em symmetric space}.

For a given pair of symmetric function spaces $E,F$ on $(0,\infty)$, the norm in $E(\M)\cap F(\M)$ (also denoted by $(E\cap F)(\M)$) is defined by setting
$$\|x\|_{E(\M)\cap F(\M)}=\max\{\|x\|_{E(\M)},\|x\|_{F(\M)}\},\quad x\in E(\M)\cap F(\M).$$
The norm in $E(\M)+F(\M)$ (also denoted by $(E+F)(\M)$) is defined by setting
$$\|x\|_{E(\M)+ F(\M)}=\inf\{\|y\|_{E(\M)}+\|z\|_{F(\M)}:\ x=y+z,\ y\in E(\M),\ z\in F(\M)\}.$$

The following lemma is elementary (see \cite[Lemma I.3.3]{Krein1982}).

\begin{lemma}\label{symm embedding lem} If $E,F$ are two symmetric function spaces on $(0,\infty)$ such that $F\subset E$, then $F$ is embedded in $E$ continuously.
\end{lemma}

The proof of the following lemma can be found in \cite[Theorem II.4.1]{Krein1982}.

\begin{lemma}\label{intermidiate lem} For every symmetric function space $E$ on $(0,\infty),$ we have
$$L_1\cap L_\infty\subset E\subset L_1+L_\infty$$
with continuous embeddings.
\end{lemma}

Lemma \ref{intermidiate lem} can be extended to symmetric spaces associated with the von Neumann algebra $\M$, namely, for any symmetric function space $E$ on $(0,\infty)$, we have
$$(L_1\cap L_\infty)(\M)\subset E(\M)\subset (L_1+L_\infty)(\M)$$
with continuous embedding (see \cite[p. 395]{DPS2023}).

Let $E^0$ be the closure of $L_1\cap L_{\infty}$ in $E.$ It follows from \cite[p.225]{dodds2014normed} that
$$E^0(\M)=\overline{\mathcal{F}(\M,\tau)}^{E(\M)}=\overline{(L_1\cap L_\infty)(\M)}^{E(\M)}.$$

\subsection{Gelfand-Naimark equivalence}

The following results are part of the standard Gelfand-Naimark equivalence theory between the category of compact topological spaces and the category of unital commutative $C^{\ast}$-algebras. For a commutative $C^*$-algebra $\mathcal{A}$, we let ${\rm Spec}(\mathcal{A})$ denote the {\em spectrum} of $\mathcal{A}$, namely, the set of all characters of $\mathcal{A}$.


\begin{lemma}\label{Gelfand lem} If $\mathcal{A}$ is a separable commutative unital $C^*$-algebra generated by its projections, then ${\rm Spec}(\mathcal{A})$ is a totally disconnected compact metrizable space.
\end{lemma}

\begin{lemma}\label{embedding lem}
Let $X,Y$ be compact metrizable spaces. There exists a $\ast$-monomorphism $$\iota:C(Y)\to C(X)$$ if and only if there exists a (unique) continuous surjection $\pi:X\to Y$ such that
\begin{equation}\label{pull-back}
\iota(f)=f\circ \pi,\quad f\in C(Y).
\end{equation}
\end{lemma}

\subsection{Notations for $n$-tuples of operators and spaces}\label{n-tuples subsec}

Let $\alpha=(\alpha(j))_{j=1}^n,\beta=(\beta(j))_{j=1}^n\in\M^n$ be commuting self-adjoint $n$-tuples of operators, $a\in \M,$ $\xi\in \mathbb{C}$, $\psi:C^\ast(\alpha)\to\M$ a $\ast$-homomorphism, and $f$ a bounded Borel function on $\mathbb{R}$, we write
$$\alpha+\beta=(\alpha(j)+ \beta(j))_{j=1}^n,\quad \xi\alpha=(\xi\alpha(j))_{j=1}^n,$$
$$\alpha\cup\beta=(\alpha(1),\ldots,\alpha(n),\beta(1),\ldots,\beta(n)),$$
$$[a,\alpha]=([a,\alpha(j)])_{j=1}^n,\quad \psi(\alpha)=(\psi(\alpha(j)))_{j=1}^n,\quad f(\alpha)=(f(\alpha(j)))_{j=1}^n,$$
where $f(\alpha(j))$ is the Borel functional calculus of $\alpha(j)$ for each $1\le j\le n.$ Sometimes it is convenient to consider $\alpha$ as a set, and in this case by writing $b\in\alpha$ we mean that $b\in \{\alpha(1),\ldots,\alpha(n)\}$.

For convenience, we write ${\rm Spec}(C^\ast(\alpha))={\rm Spec}(\alpha)$. The spectral measure of $\alpha$ is denoted by $e^{\alpha}:\mathfrak{B}(\mathbb{R}^n)\to \mathcal{P}(\M),$ where $\mathfrak{B}(\mathbb{R}^n)$ represents the collection of all Borel sets in $\mathbb{R}^n.$ For every point $\lambda\in \mathbb{R}^n$, the projection $e^{\alpha}(\{\lambda\})$ is called an {\em eigenprojection of $\alpha$} whenever it is non-zero.

Suppose $\varPhi=(E_1,\ldots,E_n)$ is an $n$-tuple of symmetric function spaces on $(0,\infty)$. Suppose $F$ is another symmetric function space on $(0,\infty)$. We write
$$\varPhi\cap F=(E_1\cap F,\ldots,E_n\cap F),\quad \|\alpha\|_{\varPhi\cap F}=\max_{1\le j\le n}\|\alpha(j)\|_{E_j\cap F}.$$

\section{Quasicentral modulus}\label{qm sec}

Recall that quasicentral modulus is defined in Section \ref{intro sec}. We collect some useful properties of quasicentral modulus in this section.

In the following lemma, we show that if $p\in\M$ is a projection that reduces $\alpha,$ then $k_{\varPhi(p\M p)}(p\alpha)\le k_{\varPhi(\M)}(\alpha).$

\begin{lemma}\label{qm vanishes on subspace} Let $\M$ be a von Neumann algebra with a faithful normal semifinite trace $\tau$. Let $\alpha\in\M^n$ be a commuting self-adjoint $n$-tuple. Suppose $\varPhi=(E_1,\ldots,E_n)$ is an $n$-tuple of symmetric function spaces on $(0,\infty).$ Suppose $p\in\M$ is a projection that commutes with $\alpha$. We have $k_{\varPhi(p\M p)}(p\alpha)\le k_{\varPhi(\M)}(\alpha).$
\end{lemma}
\begin{proof} Without loss of generality, we may assume that $k_{\varPhi(\M)}(\alpha)<\infty.$

By the definition of $k_{\varPhi(\M)},$ we have
$$\inf_{\stackrel{r\geq a}{r\in \mathcal{F}_1^+(\M)}}\|[r,\alpha]\|_{\varPhi(\M)}\le k_{\varPhi(\M)}(\alpha),\quad a\in \mathcal{F}_1^+(\M).$$
Hence, for every $\varepsilon>0$ and every $a\in\mathcal{F}_1^+(\M)$, there exists $r\in \mathcal{F}_1^+(\M)$ such that $r\geq a$ and 
$$\|[r,\alpha]\|_{\varPhi(\M)}<k_{\varPhi(\M)}(\alpha)+\varepsilon.$$
If $a\in\mathcal{F}_1^+(p\M p),$ then $r':=prp\geq pap=a$ and $[r',p\alpha]=p[r,\alpha]p.$ Hence, for every $\varepsilon>0$ and for every $a\in\mathcal{F}_1^+(p\M p),$ there exists $r'\in \mathcal{F}_1^+(p\M p)$ such that $r'\geq a$ and
$$\|[r',p\alpha]\|_{\varPhi(p\M p)}<k_{\varPhi(\M)}(\alpha)+\varepsilon.$$
By the definition of $k_{\varPhi(p\M p)},$ we have $k_{\varPhi(p\M p)}(p\alpha)<k_{\varPhi(\M)}(\alpha)+\varepsilon.$ Since $\varepsilon>0$ is arbitrarily small, the assertion follows.
\end{proof}

The following lemma is technical. It allows us to pass from the approximation in the norm $\|\cdot\|_{\varPhi(\M)}$ to the approximation in the norm $\|\cdot\|_{(\varPhi\cap L_\infty)(\M)}.$

\begin{lemma}\label{quasicentral modulus on bounded part} Let $\M$ be a von Neumann algebra with a faithful normal semifinite trace $\tau$ and $\alpha\in \M^n$ a commuting self-adjoint $n$-tuple. Let $\varPhi=(E_1,\ldots,E_n)$ be an $n$-tuple of symmetric function spaces on $(0,\infty).$ If $k_{\varPhi(\M)}(\alpha)=0$, then $k_{(\varPhi\cap L_{\infty})(\M)}(\alpha)=0.$
\end{lemma}
\begin{proof} Fix $\e>0.$ There exist a family $\{p_k\}_{k=1}^m$ of pairwise orthogonal spectral projections of $\alpha$ and a family of real numbers $\{\lambda_{j,k}:1\le j\le n,1\le k\le m\}$ such that $\sum_{k=1}^m p_k=\mathbf{1}$ and
\begin{equation}\label{dhzk eq0}
\|p_k\alpha(j)-\lambda_{j,k}p_k\|_{\M}\leq\frac{\e}{2},\quad 1\le j\le n,\ 1\le k\le m.
\end{equation}

Fix $a\in \mathcal{F}_1^+(\M)$ and set $q=\mathfrak{l}(a).$ Clearly, $q$ is a $\tau$-finite projection and $a\leq q.$

Set
$$q_k=\mathfrak{l}(p_kq),\quad 1\le k\le m.$$
Since $q_k\le p_k$ for $1\le k\le m$, it follows that $\{q_k\}_{k=1}^m$ are pairwise orthogonal.

Since $p_k$ commutes with $\alpha,$ it follows from Lemma \ref{qm vanishes on subspace} that 
$$k_{\varPhi(p_k\M p_k)}(p_k\alpha)=0.$$
Hence, by Definition \ref{quasicentral modulus def} we have 
$$\inf_{\stackrel{r\ge q_k}{r\in \mathcal{F}_1^+(p_k\M p_k)}}\|[r,p_k\alpha]\|_{\varPhi(p_k\M p_k)}=0.$$
This allows us to choose an operator $r_k\in \mathcal{F}_1^+(p_k\M p_k)$ such that $q_k\le r_k\le p_k$ and
\begin{equation}\label{commutator on small space}
\|[r_k,p_k\alpha]\|_{\varPhi(p_k\M p_k)}\leq\frac{\e}{m}.
\end{equation}
For $1\le j\le n$, noting that $r_k\le p_k$ and that $p_k$ commutes with $\alpha(j)$, we write
$$[r_k,\alpha(j)]=r_kp_k\alpha(j)-\alpha(j)p_kr_k=r_k(p_k\alpha(j)-\lambda_{j,k}p_k)-(p_k\alpha(j)-\lambda_{j,k}p_k)r_k.$$
By the triangle inequality and \eqref{dhzk eq0},
\begin{equation}\label{commutator Linfty norm}
\|[r_k,\alpha(j)]\|_{\M}\le 2\|p_k\alpha(j)-\lambda_{j,k}p_k\|_{\M}\le\e,\quad 1\le j\le n.
\end{equation}
Again noting that $r_k\le p_k$ and that $p_k$ commutes with $\alpha(j)$, we write
$$[r_k,p_k\alpha(j)]=[r_k,\alpha(j)],\quad 1\le j\le n.$$
Thus, by \eqref{commutator on small space} we have
\begin{equation}\label{commutator L1 norm}
\|[r_k,\alpha]\|_{\varPhi(\M)}\le\frac{\e}{m}.
\end{equation}

Let $r=\sum_{k=1}^m r_k.$ Since $r_k\leq p_k$ for $1\leq k\leq m$ and since $\{p_k\}_{k=1}^m$ are pairwise orthogonal projections, it follows that $r\in \mathcal{F}_1^+(\M).$ We have
$$q=\mathfrak{l}(q)=\mathfrak{l}((\sum_{1\le k\le m} p_k)q)\overset{L.\ref{left support lemma}}{\le} \bigvee_{1\le k\le m}\mathfrak{l}(p_kq)=\bigvee_{1\le k\le m}q_k=\sum_{k=1}^mq_k\le\sum_{k=1}^mr_k=r.$$
Since also $a\leq q,$ it follows that $a\leq r.$

For $1\le k\le m$, since $r_k\le p_k$ and since $p_k$ commutes with $\alpha,$ it follows that
$$[r_k,\alpha]=p_k[r_k,\alpha]p_k.$$
Thus,
$$[r,\alpha]=\sum_{k=1}^m[r_k,\alpha]=\sum_{k=1}^mp_k[r_k,\alpha]p_k.$$
Therefore,
$$\|[r,\alpha]\|_{\M}=\max_{1\leq k\leq m}\|p_k[r_k,\alpha]p_k\|_{\M}\leq\max_{1\leq k\leq m}\|[r_k,\alpha]\|_{\M}\stackrel{\eqref{commutator Linfty norm}}{\leq}\e,$$
$$\|[r,\alpha]\|_{\varPhi(\M)}\le \sum_{1\le k\le m} \|[r_k,\alpha]\|_{\varPhi(\M)}\stackrel{\eqref{commutator L1 norm}}{\le}\sum_{k=1}^m\frac{\e}{m}=\e.$$
Hence,
\begin{equation}\label{commutator estimate 2}
\|[r,\alpha]\|_{(\varPhi\cap L_{\infty})(\M)}\le \e.
\end{equation}
Since $r\geq a,$ it follows that
$$\inf_{\stackrel{r\ge a}{r\in \mathcal{F}_1^+(\M)}}\|[r,\alpha]\|_{(\varPhi\cap L_{\infty})(\M)}\leq\varepsilon.$$
Since $\varepsilon>0$ is arbitrarily small, it follows that
$$\inf_{\stackrel{r\ge a}{r\in \mathcal{F}_1^+(\M)}}\|[r,\alpha]\|_{(\varPhi\cap L_{\infty})(\M)}=0.$$
Since $a\in \mathcal{F}_1^+(\M)$ is arbitrary, the assertion follows from the definition of $k_{(\varPhi\cap L_\infty)(\M)}$.
\end{proof}

In the lemma below, we show that the quasicentral modulus with respect to the $n$-tuple $\varPhi=(L_\infty)^n$, i.e., $k_{\M}$, vanishes for all commuting self-adjoint $n$-tuples. 

\begin{lemma}\label{5.2 pre lem} Let $\M$ be a von Neumann algebra with a faithful normal semifinite trace $\tau$. Let $\alpha\in\mathcal{M}^n$ be a commuting self-adjoint $n$-tuple. For every $\tau$-finite projection $q\in\M$ and for every $\e>0$, there is a $\tau$-finite projection $p\ge q$ such that $\|[p,\alpha]\|_{\M}\le \e.$ In particular, $k_{\M}(\alpha)=0.$
\end{lemma}
\begin{proof} For every $\e>0,$ there exist a family $\{e_k\}_{k=1}^m$ of pairwise orthogonal spectral projections of $\alpha$ and a family of real numbers $\{\lambda_{j,k}:1\le j\le n,1\le k\le m\}$ such that $\sum_{k=1}^m e_k=\mathbf{1}$ and
\begin{equation}\label{spectral approximation}
\|e_k\alpha(j)-\lambda_{j,k}e_k\|_{\M}\leq\frac{\e}{2},\quad 1\le j\le n,\ 1\le k\le m.
\end{equation}

Let $q\in\M$ be a $\tau$-finite projection. Set
$$q_k=\mathfrak{l}(e_kq),\quad 1\le k\le m.$$
Note that $q_k\le e_k$ for $1\le k\le m.$ Thus, $q_{k_1}q_{k_2}=0$ if $k_1\ne k_2.$  

We have
$$q=\mathfrak{l}((\sum_{k=1}^m e_k)q)\stackrel{L.\ref{left support lemma}}{\le}\bigvee_{1\le k\le m}\mathfrak{l}(e_kq)=\sum_{k=1}^m q_k.$$
Set
$$p=\sum_{k=1}^mq_k.$$
For $1\le j\le n$ and $1\le k\le m$ we have
\begin{multline*}
[q_k,\alpha(j)]=q_k\alpha(j)-\alpha(j)q_k=q_ke_k\alpha(j)-\alpha(j)e_kq_k\\
=q_ke_k\alpha(j)-q_k\lambda_{j,k}e_k+\lambda_{j,k}e_kq_k-\alpha(j)e_kq_k=q_k(e_k\alpha(j)-\lambda_{j,k}e_k)+(\lambda_{j,k}e_k-\alpha(j)e_k)q_k.
\end{multline*}
By the triangle inequality,
\begin{equation}\label{commutator estimate uniform}
\|[q_k,\alpha(j)]\|_{\M}\le 2\|e_k\alpha(j)-\lambda_{j,k}e_k\|_{\M}\overset{\eqref{spectral approximation}}{\le} \e.
\end{equation}
We have 
$$[p,\alpha(j)]=\sum_{k=1}^m [q_k,\alpha(j)]=\sum_{k=1}^m e_k[q_k,\alpha(j)]e_k.$$
Thus, 
$$\|[p,\alpha(j)]\|_{\M}\le \max_{1\le k\le m}\|[q_k,\alpha(j)]\|_{\M}\overset{\eqref{commutator estimate uniform}}{\le} \e,\quad 1\le j\le n.$$
This proves the first assertion.

Fix $\e>0$ and $a\in \mathcal{F}_1^+(\M,\tau).$ Set $q=\mathfrak{l}(a).$ We have that  $q\in \mathcal{P}_f(\M,\tau)$ and
$$a\le \|a\|_{\M}\mathfrak{l}(a)\le \mathfrak{l}(a)=q.$$
By the first assertion of the lemma, there exists a $p\in \mathcal{P}_f(\M,\tau)$ such that $p\ge q$ (and thus $p\ge a$) and $\|[p,\alpha]\|_{\M}\le\e.$ Obviously, $p\in \mathcal{F}_1^+(\M)$. Thus,
$$\inf_{\stackrel{r\ge a}{r\in \mathcal{F}_1^+(\M)}}\|[r,\alpha]\|_{\M}\le \|[p,\alpha]\|_{\M}\le \e.$$
Since $\e>0$ is arbitrary, it follows that
$$\inf_{\stackrel{r\ge a}{r\in \mathcal{F}_1^+(\M)}}\|[r,\alpha]\|_{\M}=0.$$
Since the choice of $a\in \mathcal{F}_1^+(\M)$ is arbitrary, it follows from the definition of $k_{\M}$ that $k_{\M}(\alpha)=0.$
\end{proof}

\begin{lemma}\label{replacement for 5.2} Let $\M$ be a $\sigma$-finite von Neumann algebra with a faithful normal semifinite trace $\tau$. Let $\alpha\in\mathcal{M}^n$ be a commuting self-adjoint $n$-tuple. There exists an increasing sequence $\{p_k\}_{k\in\mathbb{Z}_+}\subset\mathcal{P}_f(\mathcal{M},\tau)$ such that $p_k\uparrow \mathbf{1}$ and 
$$\|[p_k,\alpha]\|_{\mathcal{M}}\to0,\quad k\to\infty.$$
\end{lemma}
\begin{proof} Since $\M$ is $\sigma$-finite, we may choose an increasing sequence $\{q_k\}_{k\in \mathbb{Z}_+}$ of $\tau$-finite projections such that $q_k\uparrow \mathbf{1}$. Let $p_{-1}=0$. Next, by induction, we construct a sequence $\{p_k\}_{k\in \mathbb{Z}_+}$ of $\tau$-finite projections such that
$$p_k\ge q_k\vee p_{k-1},\quad \|[p_k,\alpha]\|_{\M}\le (k+1)^{-1},\quad k\ge0.$$

By Lemma \ref{5.2 pre lem} (taken with $q=q_0$), there exists a $\tau$-finite projection $p_0\ge q_0$ such that $\|[p_0,\alpha]\|_{\M}\le 1.$ This provides the base for induction.

Suppose that $\{p_j\}_{j=0}^{k-1}$ are already constructed. By Lemma \ref{5.2 pre lem} (taken with $q=q_k\vee p_{k-1}$), there exists a $\tau$-finite projection $p_k\ge q_k\vee p_{k-1}$ such that $\|[p_k,\alpha]\|_{\M}\le (k+1)^{-1}.$ This provides the step for induction.

Since ${\bf 1}\ge p_k\ge q_k$ for each $k\ge1$ and $q_k\uparrow {\bf 1}$, it follows that $p_k\uparrow {\bf 1}$ (see \cite[Lemma B.4]{BSZZ2023}). This completes the proof.
\end{proof}

In the lemma below, we prove that $k_{\varPhi(\M)}(\delta)=0$ for any diagonal $n$-tuple $\delta\in\M^n$ and any given $n$-tuple $\varPhi$ of symmetric function spaces on $(0,\infty)$.

\begin{lemma}\label{diagonal has zero k} Let $\M$ be a $\sigma$-finite von Neumann algebra with a faithful normal semifinite trace $\tau$. Suppose $\varPhi=(E_1,\ldots,E_n)$ is an $n$-tuple of symmetric function spaces on $(0,\infty)$. Let $\delta\in\M^n$ be a diagonal $n$-tuple. We have $k_{\varPhi(\M)}(\delta)=0$.
\end{lemma}
\begin{proof} Recall that for every symmetric function space $E$ on $(0,\infty)$, $L_1\cap L_\infty$ is embedded continuously in $E$ (see Lemma \ref{intermidiate lem}). Thus, there exists a constant $C_{\varPhi}>0$ such that $\|\cdot\|_{\varPhi}\le C_{\varPhi}\|\cdot\|_{L_1\cap L_\infty}$ on $L_1\cap L_\infty$. Hence, by Definition \ref{quasicentral modulus def}, 
$$k_{\varPhi(\M)}(\delta)\leq C_{\varPhi}k_{(L_1\cap L_\infty)(\M)}(\delta).$$
Since $L_1\not\subset L_\infty$, it follows from \cite[Proposition 3.8]{BSZZ2023} that $k_{L_1(\M)}(\delta)=0$. Thus, Lemma \ref{quasicentral modulus on bounded part} immediately yields that $k_{(L_1\cap L_\infty)(\M)}(\delta)=0.$ Hence, $k_{\varPhi(\M)}(\delta)=0.$
\end{proof}

\section{Construction of the approximate identity map}\label{construction section}

Recall the approach used in the proof of \cite[Theorem 1.2]{BSZZ2023}, where $\M$ is an infinite $\sigma$-finite factor and $\alpha\in\M^n$ is a commuting self-adjoint $n$-tuple. Choose a sequence of pairwise orthogonal projections $\{p_k\}_{k\in \mathbb{N}}\subset \mathcal{P}(\M)$ such that $p_k\sim \mathbf{1}$ for each $k\in \mathbb{N}$ and $\sum_{k\in \mathbb{N}}p_k=\mathbf{1}.$ Choose a separating family $\{\psi_k\}_{k\in \mathbb{N}}$ of characters of $C^\ast(\alpha)$ (such a family exists since the spectrum of $C^{\ast}(\alpha)$ is compact and metrizable, hence, separable), and set
\begin{equation}\label{construction in the factor case}
\psi(b)=\sum_{k\in \mathbb{N}}\psi_k(b)p_k,\quad b\in C^\ast(\alpha).
\end{equation}
Clearly, the map $\psi:C^\ast(\alpha)\to\M$ constructed in \eqref{construction in the factor case} is a $\ast$-homomorphism. It is worth noting that the selection of the characters $\{\psi_k\}_{k\in \mathbb{N}}$ above is quite arbitrary, and the only condition is that this family separates $C^\ast(\alpha)$, which ensures that $\psi$ is injective (in other words, $\psi$ is a $\ast$-monomorphism). 

A key step in the proof of \cite[Theorem 1.2]{BSZZ2023} is showing that the $\ast$-monomorphism $\psi:C^\ast(\alpha)\to\M$ constructed in \eqref{construction in the factor case} satisfies $\psi\sim_{\M}{\rm Id}_{C^{\ast}(\alpha)}$ (see Section \ref{intro sec} for this notation). However, if $\M$ is not a factor, the $\ast$-monomorphism $\psi$ constructed in \eqref{construction in the factor case} may not satisfy $\psi\sim_{\M}{\rm Id}_{C^{\ast}(\alpha)}.$

%

Due to the above reason, we need to construct a $\ast$-monomorphism $\psi:C^\ast(\alpha)\to\M$ by a formula which roughly resembles \eqref{construction in the factor case} and such that $\psi\sim_{\M}{\rm Id}_{C^{\ast}(\alpha)}.$ This is precisely the aim of this section. This is done by replacing the separating family of characters $\{\psi_k\}_{k\in \mathbb{N}}$ on $C^\ast(\alpha)$ in \eqref{construction in the factor case} with a separating family of centre-valued $\ast$-homomorphisms $\{\psi_k:C^\ast(\alpha)\to \mathcal{Z}(\M)\}_{k\in \mathbb{N}}$. When $\M$ is a factor, i.e., $\mathcal{Z}(\M)=\mathbb{C}\mathbf{1}$, we retrieve a separating family of characters on $C^\ast(\alpha)$.

Let $\M$ be a properly infinite $\sigma$-finite von Neumann algebra. In order to construct a separating family of centre-valued $\ast$-homomorphisms $\{\psi_k:C^\ast(\alpha)\to \mathcal{Z}(\M)\}_{k\in \mathbb{N}}$, we actually construct $\psi_k$ on a larger $C^\ast$-algebra $C^\ast(G\cup Z)\supset C^\ast(\alpha)$, where $G$ is a family of spectral projections of $\alpha$ built on a dyadic partition of $\mathbb{R}^n$, $Z$ is the family of central supports for the elements in $G$ (see Construction \ref{section2 main construction}), and we employ the following technical assumption:  
$$WZ^\ast(\alpha)\cap \mathcal{P}_f(\M)=\{0\}.$$
(Recall that $WZ^\ast(\alpha)$ is the von Neumann subalgebra in $\M$ generated by $W^\ast(\alpha)$ and $\mathcal{Z}(\M)$.) This allows us to derive below a simple yet useful lemma.

\begin{lemma}\label{properly infinite lemma 1} Let $\M$ be a properly infinite von Neumann algebra. If $\alpha\in\M^n$ is a commuting self-adjoint $n$-tuple such that 
\begin{equation}\label{finite assumption}
WZ^\ast(\alpha)\cap \mathcal{P}_f(\M)=\{0\},
\end{equation} then every non-zero projection in $WZ^\ast(\alpha)$ is properly infinite.
\end{lemma}
\begin{proof}  Suppose $0\ne p\in WZ^\ast(\alpha)$ is a projection. There exists a central projection $z\in \M$ such that $zp$ is properly infinite and that $(\mathbf{1}-z)p$ is finite (see e.g. \cite[Proposition V.1.19]{T1}). Obviously, $(\mathbf{1}-z)p\in WZ^\ast(\alpha)$. Thus, by \eqref{finite assumption}, $(\mathbf{1}-z)p=0.$ Hence, $p=zp$ is properly infinite.
\end{proof}


\begin{lemma}\label{properly infinite sequence lemma} Let $\M$ be a properly infinite von Neumann algebra. If $\{p_k\}_{k\in \mathbb{N}}$ is a sequence of pairwise orthogonal properly infinite projections in $\M$ such that $\mathbf{1}=\sum_{k\in \mathbb{N}}p_k$, then $$WZ^\ast(\{p_k\}_{k\in \mathbb{N}})\cap \mathcal{P}_f(\M)=\{0\}.$$
\end{lemma}
\begin{proof} Every projection in $WZ^\ast(\{p_k\}_{k\in \mathbb{N}})$ is of the shape $\sum_{k\in \mathbb{N}}p_kf_k$ where $f_k\in \mathcal{P}(\mathcal{Z}(\M))$ for all $k\in \mathbb{N}.$ Since $p_k$ is properly infinite, it follows that $p_kf_k$ is either infinite or zero. Hence, every projection in $WZ^\ast(\{p_k\}_{k\in \mathbb{N}})$ is either infinite or zero.
\end{proof}

We need the following topological lemmas for totally disconnected compact metrizable spaces. The first one is elementary. We could not locate a complete proof of it in the literature and decided to prove it for convenience of the reader.

\begin{lemma}\label{clopen lem} Let $X$ be a totally disconnected compact metrizable space. There exists a countable base of topology in $X$ which consists of clopen sets.
\end{lemma}
\begin{proof} Let us equip $X$ with a metric $\rho_X.$

As $X$ is metric and totally disconnected, it has a base of topology consisting of clopen sets (see e.g., \cite[Theorem 29.7]{Willard1970}). We show that there exists a countable base which consists of clopen sets.

For every $m\in \mathbb{N},$ $X=\cup_{x\in X} U_x^m$ where $U_x^m$ is the open ball centred at $x\in X$ with diameter $1/m$. Each $U_x^m$ can be written as a union of clopen subsets. Hence, there is a family $\{A_i^m\}_{i\in I_m}$ of clopen subsets of $X$ such that $X=\cup_{i\in I_m}A_i^m$ and that ${\rm diam}(A_i^m)<1/m$ for $i\in I_m$. As $X$ is compact, we may choose a finite sub-cover from $\{A_i^m\}_{i\in I_m}$. For simplicity, let us just say that $I_m$ is finite. 

Set 
\begin{equation}\label{def of Gamma class}
\Gamma=\cup_{m\ge1}\{A^m_i\}_{i\in I_m}.
\end{equation} 

Let $U\subset X$ be open and let $x_0\in U.$ There exists $m\in\mathbb{N}$ such that 
$$B(x_0,\frac1m):=\{x\in X:\ \rho_X(x_0,x)<\frac1m\}\subset U.$$ 
Choose $i\in I_m$ such that $x_0\in A^m_i$, and thus
$$A^m_i\subset B(x_0,\frac1m)\subset U.$$
Hence, $\Gamma$ is a base of topology in $X.$
\end{proof}

The next lemma follows from the fact that a continuous mapping of a compact metrizable space has a Borel inverse mapping \cite[Theorem 6.9.7]{Bogachev}. 

\begin{lemma}\label{app1}
Let $X$ be a totally disconnected compact metrizable space and let $Y$ be a Hausdorff topological space. Let $\pi:X\to Y$ be a continuous surjective map. There exists a family of Borel mappings  $\{\pi_k:Y\to X\}_{k\in \mathbb{N}}$ such that
\begin{enumerate}[\rm (i)]
\item\label{point preserve} $\pi\circ\pi_k={\rm Id}_Y;$
\item\label{point equiv} for every open $A\subset X,$ we have
$$A\subset \cup_{k\in\mathbb{N}}(\pi_k\circ\pi)^{-1}(A);$$
\item\label{functions point ineq} for every $0\le f\in C(X),$ we have
$$\sup_{k\in\mathbb{N}}f\circ\pi_k\circ\pi\geq f.$$
\end{enumerate}
\end{lemma}
\begin{proof} By Lemma \ref{clopen lem}, there exists $\Gamma=\{X_i\}_{i\in\mathbb{N}}$ --- a countable base of topology in $X$ such that each $X_i$ is clopen. Since each $X_i$ is closed and since $X$ is compact, it follows that each $X_i$ is compact. For each $i\in\mathbb{N},$ it follows from \cite[Theorem 6.9.7]{Bogachev} that there exists a Borel mapping $\varpi_i:\pi(X_i)\to X_i$ such that
\begin{equation}\label{bogachev equality} 
\pi\circ \varpi_i ={\rm Id}_{\pi(X_i)}.
\end{equation}

Since $\pi$ is surjective and $\cup_{i\in \mathbb{N}}X_i=X$, it follows that 
\begin{equation}\label{surjective eq}
\cup_{i\in \mathbb{N}}\pi(X_i)=Y.
\end{equation}
Let $\mathfrak{S}$ denote the group of all {\it finite} permutations of $\mathbb{N}.$ For every $\sigma\in\mathfrak{S},$ define the mapping $\pi_{\sigma}:Y\to X$ by the formula
$$\pi_{\sigma}(y)=\varpi_{\sigma(i)}(y),\quad y\in\pi(X_{\sigma(i)})\backslash\cup_{j<i}\pi(X_{\sigma(j)}),\quad i\in\mathbb{N}.$$
We aim to verify the assertions of the lemma for the collection $\{\pi_{\sigma}\}_{\sigma\in\mathfrak{S}}.$

Fix $\sigma\in\mathfrak{S}.$ For every $y\in Y,$ by \eqref{surjective eq}, there exists a unique $i\in\mathbb{N}$ such that
$$y\in\pi(X_{\sigma(i)})\backslash\cup_{j<i}\pi(X_{\sigma(j)}).$$
We have
$$\pi(\pi_{\sigma}(y))=\pi(\varpi_{\sigma(i)}(y))\stackrel{\eqref{bogachev equality}}{=}y.$$
Since $\sigma\in \mathfrak{S}$ is  arbitrary, this yields \eqref{point preserve}.

We now turn to the assertion \eqref{point equiv}. Let $i\in\mathbb{N}.$ Choose (any) $\sigma\in\mathfrak{S}$ such that $\sigma(1)=i.$ We have
$$\pi_{\sigma}(y)=\varpi_{\sigma(1)}(y)=\varpi_i(y),\quad y\in\pi(X_{\sigma(1)})=\pi(X_i).$$
Thus, $\pi_{\sigma}(\pi(X_i))\subset X_i$ or, equivalently, $X_i\subset(\pi_{\sigma}\circ\pi)^{-1}(X_i).$ Hence,
$$X_i\subset\cup_{\sigma\in\mathfrak{S}}(\pi_{\sigma}\circ\pi)^{-1}(X_i).$$

Let $A\subset X$ be an open subset. Since $\Gamma$ is a base of the topology on $X,$ it follows that there exists a subset $I_A\subset\mathbb{N}$ such that
$$A=\cup_{i\in I_A}X_i.$$
For every $i\in I_A,$ it follows from the preceding paragraph that
$$X_i\subset\cup_{\sigma\in\mathfrak{S}}(\pi_{\sigma}\circ\pi)^{-1}(X_i)\subset \cup_{\sigma\in\mathfrak{S}}(\pi_{\sigma}\circ\pi)^{-1}(A).$$
Thus,
$$A\subset\cup_{\sigma\in\mathfrak{S}}(\pi_{\sigma}\circ\pi)^{-1}(A).$$
This proves the assertion \eqref{point equiv}.  

To see \eqref{functions point ineq}, set $A_{\varepsilon}=\{x\in X:\ f(x)>\varepsilon\}$ for every $\varepsilon\ge0.$ For every $\sigma\in \mathfrak{S}$, we have
$$f\circ\pi_{\sigma}\circ\pi\geq \varepsilon\chi_{A_{\varepsilon}}\circ\pi_{\sigma}\circ\pi=\varepsilon\chi_{(\pi_{\sigma}\circ\pi)^{-1}(A_{\varepsilon})}.$$
Thus,
$$\sup_{\sigma\in\mathfrak{S}}f\circ\pi_{\sigma}\circ\pi\geq \varepsilon\sup_{\sigma\in\mathfrak{S}}\chi_{(\pi_{\sigma}\circ\pi)^{-1}(A_{\varepsilon})}=\varepsilon\chi_{\cup_{\sigma\in\mathfrak{S}}(\pi_{\sigma}\circ\pi)^{-1}(A_{\varepsilon})}.$$
Since $A_{\varepsilon}$ is open, it follows from \eqref{point equiv} that $\chi_{\cup_{\sigma\in\mathfrak{S}}(\pi_{\sigma}\circ\pi)^{-1}(A_{\varepsilon})}\ge \chi_{A_{\e}}$. Thus,
$$\sup_{\sigma\in\mathfrak{S}}f\circ\pi_{\sigma}\circ\pi\geq \varepsilon\chi_{A_{\varepsilon}}.$$
Since $\varepsilon\ge0$ is arbitrary, the assertion \eqref{functions point ineq} follows from the equality $f=\sup_{\e\ge0}\e\chi_{A_\e}.$
\end{proof}

Suppose $\nu$ is a Borel measure on $X$. We let $[f]_{\nu}$ denote the equivalence class of a bounded Borel measurable function $f:X\to \mathbb{C}$ with respect to the measure $\nu$. Define
$$L_{\infty}(X,\nu)=\{[f]_{\nu}:\ f:X\to \mathbb{C}\text{ is bounded and Borel measurable}\}.$$

Suppose $\{[f_i]_{\nu}\}_{i\in I}$ is a family of uniformly bounded elements in $L_{\infty}(X,\nu)$, we let $\vee_{i\in I} [f_i]_{\nu}$ denote the least upper bound of $\{[f_i]_{\nu}\}_{i\in I}$ in $L_{\infty}(X,\nu).$

\begin{lemma}\label{app2} Let $X,$ $Y,$ $\pi$ and $\{\pi_k\}_{k\in\mathbb{N}}$ be as in Lemma \ref{app1}. Let $\nu$ be a Borel measure on $X$ such that $\nu(A)>0$ for every non-empty open $A\subset X.$ Let $\nu\circ\pi^{-1}$ denote the measure $B\mapsto \nu(\pi^{-1}(B)),B\in \mathfrak{B}(Y).$
We have
\begin{enumerate}[{\rm (i)}]
\item\label{app2a} the natural embedding $f\mapsto [f]_{\nu}$ (respectively, $g\mapsto[g]_{\nu\circ\pi^{-1}}$) delivers an isometric $\ast$-monomorphism from $C(X)$ to $L_{\infty}(X,\nu)$ (respectively, from $C(Y)$ to $L_{\infty}(Y,\nu\circ\pi^{-1})$);
\item\label{app2b} $g\circ\pi\circ\pi_k=g$ for all $g\in C(Y)$;
\item\label{app2c} if $0\le f\in C(X)$, then
$$\vee_{k\in\mathbb{N}}[f\circ\pi_k\circ\pi]_{\nu}\geq [f]_{\nu}$$
in $L_{\infty}(X,\nu);$
\item\label{app2d} if $0\leq f\in C(X)$ is such that
$$[f\circ \pi_k]_{\nu\circ\pi^{-1}}=0,\quad \forall k\in\mathbb{N}$$
in $L_{\infty}(Y,\nu\circ\pi^{-1}),$ then $f=0.$
\end{enumerate}
\end{lemma}
\begin{proof} Every $\ast$-monomorphism of a $C^{\ast}$-algebra into another $C^\ast$-algebra is an isometry (see e.g. \cite[Corollary I.5.4]{T1}). If $f\in C(X)$ is such that $[f]_{\nu}=0,$ then $\nu(\{f\neq0\})=0.$ Since $\{f\neq0\}$ is open, it follows from the assumption on $\nu$ that $\{f\neq0\}=\varnothing.$ Hence, $f=0.$ If $g\in C(Y)$ is such that $[g]_{\nu\circ\pi^{-1}}=0,$ then 
$$(\nu\circ\pi^{-1})(\{g\ne0\})=0.$$
Since $\{g\neq0\}$ is open and, therefore, $\pi^{-1}(\{g\neq0\})$ is open, it follows from the assumption on $\nu$ that $\pi^{-1}(\{g\neq0\})=\varnothing.$ Since $\pi$ is surjective, it follows that $\{g\neq0\}=\varnothing.$ Hence, $g=0.$ This proves \eqref{app2a}

The assertion \eqref{app2b} follows immediately from Proposition \ref{app1} \eqref{point preserve}.

For every $0\le f\in C(X)$ by Lemma \ref{app1} \eqref{functions point ineq}, 
$$\vee_{k\in\mathbb{N}}f\circ\pi_k\circ\pi\geq f.$$
Since, for every sequence $\{f_k\}_{k\in\mathbb{N}}$ of Borel measurable functions, we have
$$\vee_{k\in\mathbb{N}}[f_k]_{\nu}=[\vee_{k\in\mathbb{N}}f_k]_{\nu},$$
it follows that
$$\vee_{k\in\mathbb{N}}[f\circ\pi_k\circ\pi]_{\nu}\geq [f]_{\nu}.$$
This completes the proof of \eqref{app2c}.

To see \eqref{app2d}, fix $0\leq f\in C(X)$ such that
$$[f\circ\pi_k]_{\nu\circ\pi^{-1}}=0\quad\text{ for all } k\in\mathbb{N}.$$
Namely, for each $k\in \mathbb{N}$, there exists a ($\nu\circ\pi^{-1}$)-null subset $N_k\subset Y$ such that
$$(f\circ\pi_k)(y)=0,\quad y\in Y\setminus N_k.$$
Then 
$$(f\circ\pi_k\circ\pi)(x)=(f\circ\pi_k)(\pi(x))=0,\quad x\in \pi^{-1}(Y\setminus N_k).$$
Note that $\pi^{-1}(Y\setminus N_k)=X\setminus\pi^{-1}(N_k).$ Since $N_k$ is ($\nu\circ\pi^{-1}$)-null, it follows that $\pi^{-1}(N_k)$ is $\nu$-null. Thus, 
$$[f\circ\pi_k\circ\pi]_{\nu}=0\quad\text{ for all } k\in \mathbb{N}.$$
Applying \eqref{app2c}, we obtain $0\geq [f]_{\nu}$. Since $f\geq0$ in $C(X),$ it follows that also $[f]_\nu\ge 0$ and thus $[f]_{\nu}=0.$ By \eqref{app2a}, $f=0$ in $C(X).$
\end{proof}

Suppose $X$ is a topological space, we let $\mathfrak{B}(X)$ denote the Borel $\sigma$-algebra in $X.$ Let $H$ be a Hilbert space and let $\nu:\mathfrak{B}(X)\to \mathcal{P}(B(H))$ be a spectral measure (see \cite[Definition 12.17]{Rudin1991}). Let $[f]_{\nu}$ denote the equivalence class of a bounded Borel measurable function $f:X\to \mathbb{C}$ with respect to the spectral measure $\nu$. Following \cite{Rudin1991} (see e.g. p.318 there), we define
$$L_{\infty}(X,\nu)=\{[f]_{\nu}:\ f:X\to \mathbb{C}\text{ is bounded and Borel measurable}\}.$$

For a countable commuting family of projections $G\subset \mathcal{P}(\M)$ and a countable family of central projections $Z$ in $\M$, let $\mathcal{A}=C^\ast(G\cup Z)$. Now, based on the standard Gelfand-Naimark equivalence theory and Lemma \ref{app2}, we are prepared to give the following result of the existence of a separating family $\{\psi_k\}_{k\in \mathbb{N}}$ of centre-valued $\ast$-homomorphisms on $\mathcal{A}=C^\ast(G\cup Z)$ such that each $\psi_k$ is an identity map on the $C^\ast$-subalgebra $C^\ast(Z)$.

\begin{proposition}\label{separating family prop} Let $\M$ be a von Neumann algebra. Suppose $G\subset \mathcal{P}(\M)$ is a countable commuting family of projections, and $Z\subset\mathcal{P}(\mathcal{Z}(\M))$ is a countable family of central projections in $\M$ such that $\mathbf{1}\in Z.$ Let $\mathcal{A}=C^\ast(G\cup Z)$ and $\mathcal{B}=C^\ast(Z)$. There exists a family $\{\psi_k\}_{k\in\mathbb{N}}\subset{\rm Hom}(\mathcal{A},\mathcal{Z}(\M))$ such that
\begin{enumerate}[\rm (i)]
\item\label{point identity} $\psi_k|_{\mathcal{B}}={\rm Id}_{\mathcal{B}}$;
\item\label{point faithful} for every $0\ne a\in \mathcal{A}$ there exists $k\in \mathbb{N}$ such that $\psi_k(a)\ne0$;
\item\label{point projection} $p\le \vee_{k\in \mathbb{N}} \psi_k(p)$ for every $p\in G.$
\end{enumerate}
\end{proposition}
\begin{proof} Set $X={\rm Spec}(\mathcal{A})$ and denote by $\zeta_X:\mathcal{A}\to C(X)$ the Gelfand-Naimark $\ast$-isomorphism. Set $Y={\rm Spec}(\mathcal{B})$ and denote by $\zeta_Y:\mathcal{B}\to C(Y)$ the Gelfand-Naimark $\ast$-isomorphism.

Clearly, ${\rm Span}_{\mathbb{Q}}(G\cup Z)$ (span over the field $\mathbb{Q}$) is countable and dense in $\mathcal{A}.$ Thus, $\mathcal{A}$ is separable. Since $\mathcal{A}$ is generated by its projections, by Lemma \ref{Gelfand lem}, $X$ is a totally disconnected compact metrizable space.

There exists a natural embedding $\iota:C(Y)\to C(X)$ defined by the formula $$\iota=\zeta_X\circ\zeta_Y^{-1}.$$
By Lemma \ref{embedding lem}, there exists a continuous surjection $\pi:X\to Y$ satisfying
\begin{equation}\label{pi eq}
\iota(g)=g\circ \pi,\quad g\in C(Y).
\end{equation}

By \cite[Theorem 12.22 (a)]{Rudin1991}, there exist unique spectral measures $$\nu_X:\mathfrak{B}(X)\to \mathcal{P}(W^*(G\cup Z)),\quad \nu_Y:\mathfrak{B}(Y)\to \mathcal{P}(W^\ast(Z)),$$
such that
$$a=\int_X (\zeta_X(a))(x) d\nu_X(x),\quad b=\int_Y (\zeta_Y(b))(y) d\nu_Y(y),\quad a\in\mathcal{A},\quad b\in \mathcal{B}.$$
We have $\nu_X(U)\neq0$ for every non-empty open $U\subset X$, and $\nu_Y(V)\neq0$ for every non-empty open $V\subset Y$ (see \cite[Theorem 12.22 (d)]{Rudin1991}).

Define the maps $\Phi_X:L_{\infty}(X,\nu_X)\to W^*(G\cup Z)$ and $\Phi_Y:L_{\infty}(Y,\nu_Y)\to W^*(Z)$  by setting
$$\Phi_X([f]_{\nu_X})=\int_X f(x)d\nu_X(x),\quad [f]_{\nu_X}\in L_{\infty}(X,\nu_X),$$
$$\Phi_Y([g]_{\nu_Y})=\int_Y g(y)d\nu_Y(y),\quad [g]_{\nu_Y}\in L_{\infty}(Y,\nu_Y).$$
By \cite[Theorem 12.21]{Rudin1991}, $\Phi_X:L_{\infty}(X,\nu_X)\to W^*(G\cup Z)$ and $\Phi_Y:L_{\infty}(Y,\nu_Y)\to W^*(Z)$ are surjective isometric $\ast$-isomorphisms.

From Lemma \ref{app2} \eqref{app2a}, the natural embedding $f\to[f]_{\nu_X}$ delivers an isometric $\ast$-monomorphism from $C(X)$ to $L_{\infty}(X,\nu_X).$ The natural embedding $g\to[g]_{\nu_Y}$ delivers an isometric $\ast$-monomorphism from $C(Y)$ to $L_{\infty}(Y,\nu_Y).$ For convenience, we treat these morphisms as the identity mappings. By construction,
$$\Phi_X|_{C(X)}=\zeta_X^{-1},\quad \Phi_Y|_{C(Y)}=\zeta_Y^{-1}.$$

For convenience, we let $\nu_X\circ\pi^{-1}$ denote the spectral measure 
$$B\mapsto \nu_X(\pi^{-1}(B)),\quad  B\in \mathfrak{B}(Y).$$
For each $g\in C(Y)$,
\begin{multline*}
\int_Y g(y)d(\nu_X\circ\pi^{-1})(y)=\int_Xg(\pi(x))d\nu_X(x)\\
=\zeta_X^{-1}(g\circ\pi)=\zeta_X^{-1}(\iota(g))=\zeta_Y^{-1}(g)=\Phi_Y(g)=\int_Y g(y)d\nu_Y(y).
\end{multline*}
By Riesz representation theorem (see e.g. \cite[Theorem 6.19]{Rudin1987}), 
\begin{equation}\label{uniqueness of spectral measure}
\nu_X\circ\pi^{-1}=\nu_Y.
\end{equation}
Thus,
\begin{equation}\label{simutaneously extension eq}
\Phi_Y([\chi_B]_{\nu_Y})=\nu_Y(B)=\nu_X(\pi^{-1}(B))=\Phi_X([\chi_{\pi^{-1}(B)}]_{\nu_X})=\Phi_X([\chi_B\circ\pi]_{\nu_X})
\end{equation}
for every Borel set $B\subset Y.$

Let $\{\pi_k:Y\to X\}_{k\in\mathbb{N}}$ be the family of Borel mappings given by Lemma \ref{app1}. For $k\in \mathbb{N}$, define a unital $\ast$-homomorphism $\psi_k:\mathcal{A}\to \mathcal{Z}(\M)$ by the formula
$$\psi_k=\Phi_Y\circ\iota_k\circ\zeta_X,$$
where $\iota_{k}:C(X)\to L_{\infty}(Y,\nu_X\circ\pi^{-1})$ is defined by setting $$\iota_k(f)=[f\circ\pi_k]_{\nu_X\circ\pi^{-1}},\quad f\in C(X).$$

For every $k\in \mathbb{N}$ and every $b\in \mathcal{B}$, we have
$$\zeta_X(b)=(\iota\circ\zeta_Y)(b)=g\circ\pi,\quad \text{ where }g=\zeta_Y(b).$$
Thus,
\begin{multline*}
\psi_k(b)=(\Phi_Y\circ\iota_k)(\zeta_X(b))=(\Phi_Y\circ\iota_k)(g\circ\pi)=\Phi_Y(\iota_k(g\circ\pi))\\
=\Phi_Y([g\circ\pi\circ\pi_k]_{\nu_X\circ\pi^{-1}})\stackrel{\text{Lemma }\ref{app2}\eqref{app2b}}{=}\Phi_Y([g]_{\nu_X\circ\pi^{-1}})=\zeta_Y^{-1}(g)=b.
\end{multline*}
This proves the assertion \eqref{point identity}.

If $a\in\mathcal{A}$ is non-zero, then $f:=\zeta_X(a^\ast a)\ne0.$ By Lemma \ref{app2} \eqref{app2d}, there exists $k\in \mathbb{N}$ such that $\iota_k(\zeta_X(a^\ast a))=[f\circ\pi_k]_{\nu_X\circ\pi^{-1}}$ is non-zero in $L_{\infty}(Y,\nu_X\circ\pi^{-1}).$ Therefore,
$$|\psi_k(a)|^2=\psi_k(a^\ast a)=\Phi_Y(\iota_k(\zeta_X(a^\ast a)))\ne0.$$
This proves the assertion \eqref{point faithful}.

We now prove the assertion \eqref{point projection}. For every $p\in G\subset \mathcal{A},$ $f:=\zeta_X(p)\in C(X).$ By definition,
$$\psi_k(p)=\Phi_Y(\iota_k(f))=\Phi_Y([f\circ \pi_k]_{\nu_X\circ\pi^{-1}})\stackrel{\eqref{simutaneously extension eq}}{=}\Phi_X([f\circ\pi_k\circ\pi]_{\nu_X}).$$
By \cite[Corollary 7.1.16]{Kadison1997_2}, every $*$-isomorphism on a von Neumann algebra is normal (i.e., order continuous). In particular, $\Phi_X$ is order continuous. Thus,
$$\vee_{k\in\mathbb{N}}\psi_k(p)=\vee_{k\in\mathbb{N}}\Phi_X([f\circ\pi_k\circ\pi]_{\nu_X})=\Phi_X(\vee_{k\in\mathbb{N}}[f\circ\pi_k\circ\pi]_{\nu_X}).$$
Clearly,
$$\vee_{k\in\mathbb{N}}[f\circ\pi_k\circ\pi]_{\nu_X}=[\vee_{k\in\mathbb{N}}(f\circ\pi_k\circ\pi)]_{\nu_X}.$$
Hence,
$$\vee_{k\in\mathbb{N}}\psi_k(p)=\Phi_X([\vee_{k\in\mathbb{N}}f\circ\pi_k\circ\pi]_{\nu_X}).$$
By Lemma \ref{app1} \eqref{functions point ineq}, we have
$$\sup_{k\in\mathbb{N}}f\circ\pi_k\circ\pi\geq f.$$
Thus,
$$\vee_{k\in\mathbb{N}}\psi_k(p)\geq\Phi_X([f]_{\nu_X})=\zeta_X^{-1}(f)=p.$$
This completes the proof of the assertion \eqref{point projection}. 
\end{proof}

The proof of the next lemma is routine and, therefore, omitted.
\begin{lemma}\label{pointwise approx lem} Let $\M$ be a von Neumann algebra. Let $n\in \mathbb{N}$ and $\alpha\in\M^n$ a commuting self-adjoint $n$-tuple. Suppose $\psi:C^*(\alpha)\to \M$ is a unital $\ast$-homomorphism. If $\{u_m\}_{m\in \mathbb{N}}$ is a sequence of unitaries in $\M$ such that 
$$\lim_{m\to\infty} \|\psi(\alpha(i))-u_m^{-1}\alpha(i)u_m\|_{\M}=0$$
for every $1\le i\le n$, then
$$\lim_{m\to\infty}\|\psi(a)-u_m^{-1}au_m\|_{\M}=0$$
for all $a\in C^*(\alpha)$. Namely, $\psi\sim_{\M}{\rm Id}_{C^{\ast}(\alpha)}.$
\end{lemma}

Now we can construct a separable commutative $C^\ast$-algebra $\mathcal{A}\supset C^\ast(\alpha)$ and a $\ast$-homomorphism $\psi:\mathcal{A}\to\M$ such that $\psi|_{C^\ast(\alpha)}\sim_{\M} {\rm Id}_{C^{\ast}(\alpha)}$.

\begin{construction}\label{section2 main construction} Let $\M$ be a properly infinite $\sigma$-finite von Neumann algebra. Let $\alpha\in\M^n$ be a commuting self-adjoint $n$-tuple. 

For every $m\in \mathbb{Z}_+$, let ${\rm At}_m$ be the collection of all cubes
$$\Big[\frac{k_1}{2^m},\frac{k_1+1}{2^m}\Big)\times\cdots\times\Big[\frac{k_n}{2^m},\frac{k_n+1}{2^m}\Big),\quad (k_1,\ldots,k_n)\in \mathbb{Z}^n.$$
Set
$$G=\cup_{m\in \mathbb{Z}_+}\{e^{\alpha}(U):U\in {\rm At}_m\},\quad Z=\{c(p):p\in G\},\quad \mathcal{A}=C^\ast(G\cup Z).$$

Let the sequence $\{\psi_k\}_{k\in \mathbb{N}}\subset {\rm Hom}(\mathcal{A},\mathcal{Z}(\M))$ be given by Proposition \ref{separating family prop}. Let $\{p_k\}_{k\in \mathbb{N}}$ be a sequence of pairwise orthogonal projections in $\M$ such that $p_k\sim \mathbf{1}$ for each $k\in \mathbb{N}$ and such that $\sum_{k\in\mathbb{N}}p_k=\mathbf{1}.$ Define the $\ast$-homomorphism  $\psi:\mathcal{A}\to\mathcal{M}$ by the formula
$$\psi(a)=\sum_{k\in \mathbb{N}}\psi_k(a)p_k,\quad a\in \mathcal{A},$$
where the series converges in the strong operator topology.
\end{construction}

We now prove that the $\ast$-homomorphism $\psi:\mathcal{A}\to\M$ in Construction \ref{section2 main construction} is faithful.

\begin{lemma}\label{faithful lemma} Let $\M$ be a properly infinite $\sigma$-finite von Neumann algebra. Let $\alpha\in\M^n$ be a commuting self-adjoint $n$-tuple. If $\mathcal{A}$ and $\psi\in {\rm Hom}(\mathcal{A},\M)$ are given by Construction \ref{section2 main construction}, then $\psi$ is faithful.
\end{lemma}
\begin{proof} Let the sequences $\{\psi_k\}_{k\in \mathbb{N}}\subset {\rm Hom}(\mathcal{A},\mathcal{Z}(\M))$, $\{p_k\}_{k\in \mathbb{N}}\subset \mathcal{P}(\M)$ be as in Construction \ref{section2 main construction}. If $a\in \mathcal{A}$ is non-zero, then by Proposition \ref{separating family prop} \eqref{point faithful}, there exists $k\in \mathbb{N}$ such that $\psi_k(a)\ne0.$ Then the right support projection $z:=\mathfrak{r}(\psi_k(a))\ne0.$ Since $\psi_k(a)\in \mathcal{Z}(\mathcal{M})$, it follows that $z\in \mathcal{Z}(\mathcal{M}).$ Since $p_k\sim\mathbf{1}$ by Construction \ref{section2 main construction}, it follows that $c(p_k)=\mathbf{1}$ and, therefore,
$$c(zp_k)=zc(p_k)=z\ne 0.$$ 
Thus, $zp_k\ne0$ or equivalently, $\psi_k(a)p_k\ne0.$ Hence, $\psi(a)\ne0.$
\end{proof}

The following lemma plays an important role in the proof of the assertion that $\psi|_{C^\ast(\alpha)}\sim_{\M}{\rm Id}_{C^{\ast}(\alpha)}$, where $\psi:\mathcal{A}\to\M$ is the $\ast$-homomorphism in Construction \ref{section2 main construction}.

\begin{lemma}\label{psip sim p} Let $\M$ be a properly infinite $\sigma$-finite von Neumann algebra. Suppose $\alpha\in\M^n$ is a commuting self-adjoint $n$-tuple such that 
\begin{equation}\label{new assumption}
WZ^\ast(\alpha)\cap \mathcal{P}_f(\M)=\{0\}.
\end{equation}
Let $G$, $Z$, $\mathcal{A}$ and $\psi$ be as in Construction \ref{section2 main construction}. We have
$$\psi(p)\sim p,\quad p\in G.$$
\end{lemma}
\begin{proof}Let the sequences $\{\psi_k\}_{k\in \mathbb{N}}\subset {\rm Hom}(\mathcal{A},\mathcal{Z}(\M)), \{p_k\}_{k\in \mathbb{N}}\subset \mathcal{P}(\M)$ be as in Construction \ref{section2 main construction}.

Let $0\ne p\in G.$  For every $k\in\mathbb{N},$ $\psi_k(p)\in\mathcal{P}(\mathcal{Z}(\mathcal{M}))$ and $c(p_k)=\mathbf{1}$ (since $p_k\sim {\bf 1}$ according to the Construction \ref{section2 main construction}). Thus,
$$c(\psi_k(p)p_k)=\psi_k(p)\cdot c(p_k)=\psi_k(p).$$
Hence,
$$c(\psi(p))=\vee_{k\in \mathbb{N}}c(\psi_k(p)p_k)=\vee_{k\in \mathbb{N}}\psi_k(p).$$
By Proposition \ref{separating family prop} \eqref{point projection}, $p\le \vee_{k\in \mathbb{N}}\psi_k(p).$ Hence, $c(\psi(p))\geq p$ and, therefore, $c(\psi(p))\geq c(p).$

By construction, $c(p)\in Z.$ Thus,
$$\psi(p)\leq\psi(c(p))=\sum_{k\geq1}\psi_k(c(p))p_k\stackrel{\text{Proposition }\ref{separating family prop} \eqref{point identity}}{=}\sum_{k\geq1}c(p)p_k=c(p).$$ 
Therefore, $c(\psi(p))=c(p).$

Recall that 
$$0\ne p\in G\subset W^\ast(\alpha)\subset WZ^\ast(\alpha).$$
By Lemma \ref{properly infinite lemma 1}, $p$ is properly infinite (this is the only place in this section where we are using the assumption $WZ^\ast(\alpha)\cap \mathcal{P}_f(\M)=\{0\}$). Since $p_k\sim\mathbf{1},$ it follows that $p_k$ is properly infinite. Since $\psi_k(p)$ is a central projection, it follows that $\psi_k(p)p_k$ is either properly infinite or $0.$ If $\psi_k(p)p_k=0$ for every $k\in\mathbb{N},$ then $\psi(p)=0,$ which is impossible by Lemma \ref{faithful lemma}. Hence, there exists $k\in\mathbb{N}$ such that $\psi_k(p)p_k$ is properly infinite. Therefore, $\psi(p)$ is properly infinite.

Since $p$ and $\psi(p)$ are both properly infinite and since $c(\psi(p))=c(p),$ it follows from \cite[Corollary 6.3.5]{Kadison1997_2} that $\psi(p) \sim p.$
\end{proof}

\begin{lemma}\label{finite spectrum lemma} Let $\M$ be a von Neumann algebra. Let $\{p_k\}_{k=1}^l$ be a family of pairwise orthogonal projections such that $\sum_{k=1}^l p_k=\mathbf{1}$. For every unital $\ast$-homomorphism $\psi:C^{\ast}(\{p_k\}_{k=1}^l)\to\M$ satisfying $\psi(p_k)\sim p_k$ for every $1\le k\le l,$ there exists a unitary $w\in\M$ such that
$$w^{-1}p_kw=\psi(p_k),\quad 1\le k\le l.$$
\end{lemma}
\begin{proof} By assumption, for each $1\le k\le l,$ there exists a partial isometry $w_k\in\M$ such that
$$w_k^{\ast}w_k=\psi(p_k),\quad w_kw_k^{\ast}=p_k.$$
Set $w=\sum_{k=1}^l w_k.$ It is clear that $w$ is a unitary and that
$$w^{-1}p_kw=\psi(p_k),\quad 1\le k\le l.$$
The assertion follows immediately.
\end{proof}

Now, we prove that $\psi|_{C^\ast(\alpha)}\sim_{\M}{\rm Id}_{C^{\ast}(\alpha)},$ where the $\ast$-homomorphism $\psi$ is given by the Construction \ref{section2 main construction}.

\begin{theorem}\label{approximation by inner hom} Let $\M$ be a properly infinite $\sigma$-finite von Neumann algebra. Let $\alpha\in\M^n$ be a commuting self-adjoint $n$-tuple satisfying $$WZ^\ast(\alpha)\cap \mathcal{P}_f(\M)=\{0\}.$$
If $\mathcal{A}$ and $\psi$ are as in Construction \ref{section2 main construction}, then
\begin{enumerate}[\rm (i)]
\item\label{psi point 2} $\psi$ is faithful;
\item\label{psi point 3} $\psi|_{C^*(\alpha)}\sim_{\M}{\rm Id}_{C^{\ast}(\alpha)}$.
\end{enumerate}
\end{theorem}
\begin{proof} The assertion \eqref{psi point 2} has been proved in Lemma \ref{faithful lemma}.

Let $\{{\rm At}_m\}_{m\in \mathbb{N}}$, $G$ and $Z$ be as in Construction \ref{section2 main construction}. For $m\in \mathbb{N}$, set
$$G_m=\{e^{\alpha}(U):U\in {\rm At}_m\},\quad Z_m=\{c(p):p\in G_m\},\quad \mathcal{A}_m=C^\ast(G_m\cup Z_m).$$
Since $G=\cup_{m\in \mathbb{N}}G_m$, $Z=\cup_{m\in \mathbb{N}}Z_m$ and $\mathcal{A}=C^\ast(G\cup Z)$ by construction, it follows that 
$$\mathcal{A}=\overline{\cup_{m=1}^\infty\mathcal{A}_m}^{\|\cdot\|_{\M}}.$$

For $m\in \mathbb{N}$, $\mathcal{A}_m\subset\mathcal{A}$ and, hence, $\psi:\mathcal{A}_m\to\mathcal{M}$ is a unital $\ast$-homomorphism. By Lemma \ref{psip sim p}, we have that $\psi(p)\sim p$ for every $p\in G_m.$ By Lemma \ref{finite spectrum lemma}, there exists a unitary element $u_m\in\mathcal{M}$ such that 
\begin{equation}\label{psi preserves equivalence}
\psi(p)=u_m^{-1}pu_m,\quad p\in G_m.
\end{equation}
For every $U\in {\rm At_m}$, let $c_U\in U$ be the centre of the cube $U$. We write $c_U=(c_U(i))_{i=1}^n$ where $c_U(i)\in \mathbb{R}$ for $1\le i\le n$. Set
$$\beta(i)=\sum_{U\in {\rm At}_m}c_U(i)e^{\alpha}(U),\quad 1\le i\le n.$$
We have that $\beta\in ({\rm Span}(G_m))^n$ and
$$\|\alpha(i)-\beta(i)\|_{\M}\le 2^{-m-1},\quad 1\le i\le n.$$ 
By \eqref{psi preserves equivalence}, $u_m^{-1}\beta(i) u_m=\psi(\beta(i))$ for every $1\le i\le n.$ 
We have
\begin{multline*}
\psi(\alpha(i))-u_m^{-1}\alpha(i)u_m=\\
=\psi(\alpha(i)-\beta(i))+\big(\psi(\beta(i))-u_m^{-1}\beta(i)u_m\big)+u_m^{-1}(\beta(i)-\alpha(i))u_m\\
=\psi(\alpha(i)-\beta(i))+u_m^{-1}(\beta(i)-\alpha(i))u_m.
\end{multline*}
Thus, for every $1\le i\le n,$
$$\|\psi(\alpha(i))-u_m^{-1}\alpha(i)u_m\|_{\M}\le 2\|\alpha(i)-\beta(i)\|_{\M}\leq 2^{-m}.$$
The assertion \eqref{psi point 3} follows now from Lemma \ref{pointwise approx lem}.
\end{proof}

\begin{remark} Let $\psi$ be defined as in Theorem \ref{approximation by inner hom}. It is worth mentioning that, if the Hilbert space $H$ is separable, then $\psi(p)\sim p$ for every projection $p\in C^\ast(\alpha)$. Indeed, by Theorem \ref{approximation by inner hom}, $\psi|_{C^\ast(\alpha)}\sim_{\M}{\rm Id}_{C^{\ast}(\alpha)}$. Then from \cite[Theorem 3]{Ding2005},  $\mathfrak{l}(\psi(a))\sim \mathfrak{l}(a)$ for any $a\in C^\ast(\alpha)$. In particular, $\psi(p)\sim p$ for every projection $p\in C^\ast(\alpha)$. This should be compared with Lemma \ref{psip sim p}.
\end{remark}

\section{Technical result}\label{technical section}

The following theorem is the main result of this section. It can be regarded as an analogue of \cite[Lemma 1.2]{V1976} (taking $\pi_r={\rm Id}_{C^{\ast}(\alpha)}$ for each $r\in \mathbb{N}$). This theorem is a key ingredient in the proof of Theorem \ref{key thm}. 

\begin{theorem}\label{isometries commutator vanishing thm} Let $\M$ be a properly infinite $\sigma$-finite von Neumann algebra with a faithful normal semifinite trace $\tau$. Let $\alpha\in\M^n$ be a commuting self-adjoint $n$-tuple such that
\begin{equation}\label{technical cond}
WZ^*(\alpha)\cap \mathcal{P}_f(\M)=\{0\}.
\end{equation}
There exists a sequence of isometries $\{v_j\}_{j\geq0}\subset\M$ such that
\begin{enumerate}[\rm (i)]
\item\label{4.1.1} $v_{j_1}^{\ast}v_{j_2}=\delta_{j_1,j_2}\mathbf{1},\quad j_1,j_2\geq0;$
\item\label{4.1.2} $[v_j,\alpha(i)]\in \mathcal{K}(\M,\tau)$ for $j\ge0,$ $1\le i\le n$;
\item\label{4.1.3} $\|[v_j,\alpha]\|_{\M}\to0$ as $j\to\infty.$
\end{enumerate}
\end{theorem}

We need some preparations to prove Theorem \ref{isometries commutator vanishing thm}.

\begin{lemma}\label{infinite projections lemma} Suppose we are in the setting of Theorem \ref{isometries commutator vanishing thm}. For arbitrary $\varepsilon>0,$ there exists a diagonal $n$-tuple $\beta\in\M^n$ such that $\|\beta-\alpha\|_{\M}\leq\varepsilon$ and 
$$WZ^{\ast}(\beta)\cap\mathcal{P}_f(\M)=0.$$
\end{lemma}
\begin{proof} Fix $\e>0.$ Let $f_\e(x)=\e\lfloor \e^{-1}x\rfloor$, $x\in \mathbb{R}$. Here, $\lfloor x\rfloor$ denotes the integer part for $x\in \mathbb{R}$. Note that
\begin{equation}\label{diagonalization function}
|f_\e(x)-x|=|\e(\lfloor \e^{-1}x\rfloor-\e^{-1}x)|\le\e.
\end{equation}
Set
$$\beta(i)=f_\e(\alpha(i)),\quad 1\leq i\leq n.$$
It is clear that $\beta$ is a commuting $n$-tuple. Since the range of the function $f_\e$ is countable, it follows that $\beta\in\M^n$ is a diagonal $n$-tuple. From \eqref{diagonalization function}, $\|\beta(i)-\alpha(i)\|_{\M}\le \e$ for $1\le i\le n$. Thus, $\|\beta-\alpha\|_{\M}\le \e.$ 

By Lemma \ref{properly infinite lemma 1}, every non-zero projection in $WZ^\ast(\alpha)$ is properly infinite.  Let $\{p_k\}_{k\in\mathbb{N}}$ be the sequence of all eigenprojections of $\beta$ and note that they are also the spectral projections of $\alpha.$ By Lemma \ref{properly infinite sequence lemma}, 
$$WZ^\ast(\{p_k\}_{k\in \mathbb{N}})\cap \mathcal{P}_f(\M)=\{0\}.$$
Clearly, $WZ^\ast(\beta)=WZ^\ast(\{p_k\}_{k\in \mathbb{N}})$. 
Hence, $WZ^\ast(\beta)\cap \mathcal{P}_f(\M)=\{0\}$.
\end{proof}

\begin{lemma}\label{diagonal pass to tensor lem} Let $\M$ be a properly infinite $\sigma$-finite von Neumann algebra with a faithful normal semifinite trace $\tau$. Let $\beta\in\M^n$ be a diagonal $n$-tuple such that $$WZ^\ast(\beta)\cap \mathcal{P}_f(\M)=0.$$
For a given $\tau$-finite projection $r\in\M,$ there exists an isometry $v\in\M$ such that
$$vv^{\ast}\leq \mathbf{1}-r,\quad \beta=v^{\ast}\beta v.$$
\end{lemma}
\begin{proof} Let $(p_l)_{l\in\mathbb{N}}$ be the family of all eigenprojections of $\beta.$ From Lemma \ref{properly infinite lemma 1}, each $p_l$ is properly infinite. Note that $\tau(\mathfrak{r}(rp_l))\le \tau(r)<\infty$ and thus $\mathfrak{r}(rp_l)$ is finite. Hence, for $l\in \mathbb{N}$, $p_l-\mathfrak{r}(rp_l)$ is properly infinite. By Lemma \ref{equivalent projections lem}, $p_l-\mathfrak{r}(rp_l)\sim p_l$. Thus, there exists a partial isometry $v_l\in\M$ such that 
$$v_l^\ast v_l=p_l,\quad v_lv_l^\ast=p_l-\mathfrak{r}(rp_l).$$

We have
$$v_{l_1}^{\ast}v_{l_2}=\delta_{l_1,l_2}p_{l_1},\quad v_{l_1}v_{l_2}^{\ast}=\delta_{l_1,l_2}(p_{l_1}-\mathfrak{r}(rp_{l_1})),\quad  v_{l_1}^{\ast}p_lv_{l_2}=\delta_{l_1,l}\delta_{l_2,l}p_l,\quad l_1,l_2,l\in \mathbb{N}.$$
Note that for $l\in \mathbb{N},$
$$r\cdot(p_l-\mathfrak{r}(rp_l))=rp_l-r\cdot \mathfrak{r}(rp_l)=rp_l-r\cdot p_l\mathfrak{r}(rp_l)=rp_l-rp_l\cdot \mathfrak{r}(rp_l)=0.$$
Thus, $v_lv_l^\ast\le \mathbf{1}-r$ for every $l\in \mathbb{N}.$

Set $v=\sum_{l\in \mathbb{N}} v_l.$ It is immediate from the above paragraph that 
$$v^{\ast}v=\mathbf{1},\quad vv^{\ast}=\sum_{l\in\mathbb{N}}v_lv_l^{\ast}\leq\mathbf{1}-r,\quad v^{\ast}p_kv=p_k,\quad k\in\mathbb{N}.$$
Thus, $v^{\ast}\beta v=\beta$ and this completes the proof.
\end{proof}

\begin{lemma}\label{left orthogonal pre lemma} Suppose we are in the setting of Theorem \ref{isometries commutator vanishing thm}. Let $p,r\in \mathcal{P}(\M),$ and let $r$ be $\tau$-finite. For every $\varepsilon>0,$ there exists a partial isometry $w\in\M$ such that $ww^{\ast}\leq \mathbf{1}-r,$ $w^{\ast}w=p$ and 
$$\|[w,\alpha]\|_{\M}\leq \varepsilon+\|[p,\alpha]\|_{\M}.$$
\end{lemma}
\begin{proof} We may assume without loss of generality that $\|\alpha\|_{\M}\leq1.$

Applying Lemma \ref{infinite projections lemma} to the $2n$-tuple $\alpha\cup\alpha^2$ where $\alpha^2=(\alpha(1)^2,\ldots,\alpha(n)^2)$, we find a diagonal $2n$-tuple $\beta\in\M^{2n}$ such that
$$\|\alpha\cup\alpha^2-\beta\|_{\M}\leq\frac{\varepsilon^2}{6},\quad WZ^{\ast}(\beta)\cap\mathcal{P}_f(\M)=\{0\}.$$
Applying Lemma \ref{diagonal pass to tensor lem} to the $2n$-tuple $\beta,$ one can find an isometry $v\in\M$ such that $vv^{\ast}\leq \mathbf{1}-r,$ and
$$\beta=v^{\ast}\beta v.$$
We have
\begin{multline*}
\|\alpha\cup \alpha^2-v^{\ast}(\alpha\cup\alpha^2)v\|_{\M}\leq \|\alpha\cup\alpha^2-\beta\|_{\M}+\|v^{\ast}\beta v-v^{\ast}(\alpha\cup\alpha^2)v\|_{\M}\\
\leq 2\|\alpha\cup\alpha^2-\beta\|_{\M}\leq\frac{\varepsilon^2}{3}.
\end{multline*}

We now set $w=vp.$ Clearly, $w$ is a partial isometry such that $w^{\ast}w=p.$ Obviously, $ww^{\ast}\le (\mathbf{1}-r)$, and for $b\in\alpha\cup \alpha^2$ we have
$$pbp-w^{\ast}bw=p(b-v^{\ast}bv)p.$$
Hence,
$$\|pbp-w^{\ast}bw\|_{\M}\leq \frac{\varepsilon^2}{3},\quad b\in\alpha\cup\alpha^2.$$

Note that
\begin{multline*}
|[w,b]|^2=(bw^{\ast}-w^{\ast}b)(wb-bw)\\
=bpb-bw^{\ast}bw-w^{\ast}bw b+w^{\ast}b^2w\\
=bpb+b\cdot (pbp-w^{\ast}bw)-bpbp\\
+(pbp-w^{\ast}bw)\cdot b-pbpb-(pb^2p-w^{\ast}b^2w)+pb^2p\\
=b\cdot (pbp-w^{\ast}bw)+(pbp-w^{\ast}bw)\cdot b\\
-(pb^2p-w^{\ast}b^2w)+(bpb-bpbp-pbpb+pb^2p)\\
=b\cdot (pbp-w^{\ast}bw)+(pbp-w^{\ast}bw)\cdot b-(pb^2p-w^{\ast}b^2w)-[p,b]^2.
\end{multline*}
By the triangle inequality, we have, for every $b\in\alpha,$
\begin{multline*}
\|[w,b]\|_{\M}^2\leq 2\|b\|_{\M}\|pbp-w^{\ast}bw\|_{\M}+\|pb^2p-w^{\ast}b^2w\|_{\M}+\|[p,b]\|_{\M}^2\\
\leq 2\cdot 1\cdot\frac{\varepsilon^2}{3}+\frac{\varepsilon^2}{3}+\|[p,b]\|_{\M}^2\leq (\varepsilon+\|[p,b]\|_{\M})^2.
\end{multline*}
Thus, 
$$\|[w,\alpha(i)]\|_{\M}\le \e+\|[p,\alpha(i)]\|_{\M},\quad 1\le i\le n.$$ 
Hence, $\|[w,\alpha]\|_{\M}\le \e+\|[p,\alpha]\|_{\M}$.
\end{proof}

\begin{lemma}\label{induction lemma} Suppose we are in the setting of Theorem \ref{isometries commutator vanishing thm}. Given a sequence of $\tau$-finite projections $\{p_{k}\}_{k\geq0}\subset\M$, and a sequence $\{\e_k\}_{k\ge0}$ of positive real numbers, there exists a sequence of partial isometries $\{w_k\}_{k\geq0}\subset\M$ such that
$$w_l^{\ast}w_k=\delta_{l,k}p_k,\quad \|[w_k,\alpha]\|_{\M}\leq \e_k+\|[p_k,\alpha]\|_{\M},\quad l,k\geq 0.$$
\end{lemma}
\begin{proof} We construct inductively a sequence of partial isometries $\{w_k\}_{k\geq0}\subset\M$ such that
$$w_l^{\ast}w_k=\delta_{l,k}p_k,\quad \|[w_k,\alpha]\|_{\M}\leq \e_k+\|[p_k,\alpha]\|_{\M},\quad 0\leq l\leq k.$$
When such a sequence is already constructed, the equality $w_l^{\ast}w_k=\delta_{l,k}p_k$ for $0\leq k\leq l$ follows by taking adjoints.

By Lemma \ref{left orthogonal pre lemma} (applied with $p=p_0,$ and $r=0$), we find a partial isometry $w_0\in\M$ such that 
$$w_0^{\ast}w_0=p_0,\quad \|[w_0,b]\|_{\M}\leq \e_0+\|[p_0,b]\|_{\M},\quad b\in\alpha.$$
Note that $p_0$ is $\tau$-finite, so $w_0$ is $\tau$-finitely supported. This yields the base of the induction.

For $k\ge1,$ suppose that the sequence $\{w_l\}_{l=0}^{k-1}$ of partial isometries is already constructed. Set
$$q_k=\bigvee_{0\leq l<k}\mathfrak{l}(w_l).$$
We have
$$\tau(q_k)\le \sum_{0\le l<k}\tau(\mathfrak{l}(w_l))=\sum_{0\leq l<k}\tau(\mathfrak{r}(w_l))=\sum_{0\leq l<k}\tau(p_l)<\infty.$$
By Lemma \ref{left orthogonal pre lemma} (applied with $p=p_k$ and $r=q_k$), we can find a partial isometry $w_k\in \M$ such that 
$$w_k^{\ast}w_k=p_k,\quad w_kw_k^{\ast}\leq \mathbf{1}-q_k,\quad \|[w_k,b]\|_{\M}\leq \e_{k}+\|[p_k,b]\|_{\M},\quad b\in\alpha.$$

Since $q_kw_k=0,$ it follows that 
$$\mathfrak{l}(w_l)\cdot w_k=0,\quad l<k,\quad b\in\alpha.$$
Thus,
$$w_l^{\ast}w_k=w_l^{\ast}\cdot \mathfrak{r}(w_l^{\ast})\cdot w_k=w_l^{\ast}\cdot \mathfrak{l}(w_l)\cdot w_k=0,\quad l<k.$$
Thus,
$$w_l^{\ast}w_k=\delta_{l,k}p_k,\quad \|[w_k,\alpha]\|_{\M}\leq \e_k+\|[p_k,\alpha]\|_{\M},\quad 0\leq l\leq k.$$
This yields the step of the induction and, hence, completes the proof.
\end{proof}


\begin{proof}[Proof of Theorem \ref{isometries commutator vanishing thm}] By Lemma \ref{replacement for 5.2}, there exists an increasing sequence $\{e_j\}_{j\in\mathbb{Z}_+}\subset\mathcal{P}_f(\mathcal{M},\tau)$ such that $e_j\uparrow \mathbf{1}$ and
\begin{equation}\label{loml eq0}
\|[e_j,\alpha]\|_{\mathcal{M}}\leq 2^{-j},\quad j\in\mathbb{Z}_+.
\end{equation}
For $j\ge0,$ set 
$$p_{j,j}=e_j,\quad p_{m,j}=e_{m}-e_{m-1},\quad m>j.$$

Let $e_{-1}=0$. Note that
$$\|[p_{m,j},\alpha]\|_{\M}\leq \|[e_m,\alpha]\|_{\M}+\|[e_{m-1},\alpha]\|_{\M}\le 2^{-m}+2^{-(m-1)}=3\cdot 2^{-m},\quad m\geq j.$$

Applying Lemma \ref{induction lemma} to the sequence $\{p_{m,j}\}_{m\geq j\geq0}\subset\mathcal{M}$ and to the sequence $\{\e_{m,j}\}_{m\ge j\ge 0}$ where $\e_{m,j}=2^{-m}$, we obtain a sequence $\{w_{m,j}\}_{m\geq j\geq0}$ of partial isometries such that 
\begin{enumerate}[{\rm a)}]
\item $w_{m,j}^{\ast}w_{m,j}=p_{m,j}$ for all $m\geq j\geq0;$
\item $w_{m_1,j_1}^{\ast}w_{m_2,j_2}=0$ for all $(m_1,j_1)\neq(m_2,j_2);$
\item for all $m\geq j\geq0$ we have
$$\|[w_{m,j},\alpha]\|_{\M}\le \e_{m,j}+\|[p_{m,j},\alpha]\|_{\M}\le 2^{-m}+3\cdot 2^{-m}=2^{-(m-2)}.$$
\end{enumerate} 

Set
$$v_j=\sum_{m\geq j}w_{m,j},\quad j\in\mathbb{Z}_+,$$
where the series converges in the strong operator topology. We have 
$$v_{j_1}^{\ast}v_{j_2}=\sum_{\substack{m_1\geq j_1\\ m_2\geq j_2}}w_{m_1,j_1}^{\ast}w_{m_2,j_2}=\delta_{j_1,j_2}\sum_{m\geq j_1}p_{m,j_1}=\delta_{j_1,j_2}\mathbf{1}.$$
Hence, the assertion \eqref{4.1.1} is proved.

For $j\in \mathbb{Z}_+$ and $1\le i\le n$, 
\begin{equation}\label{commutant with isometry eq}
[v_j,\alpha(i)]=\sum_{m\ge j}[w_{m,j},\alpha(i)].
\end{equation}
Since $w_{m,j}\in \mathcal{F}(\M,\tau)$, it follows that $[w_{m,j},\alpha(i)]\in \mathcal{F}(\M,\tau)$ for $m\ge j$. We now show that the series in the right-hand side of \eqref{commutant with isometry eq} also converges absolutely in the uniform norm $\|\cdot\|_{\M}$ and, therefore, $[v_j,\alpha(i)]\in \mathcal{K}(\M,\tau)$. By the triangle inequality, we have
$$\|[v_j,\alpha(i)]\|_{\M}\leq\sum_{m\geq j}\|[w_{m,j},\alpha(i)]\|_{\M}\leq\sum_{m\geq j}2^{-(m-2)}=2^{-(j-3)},\quad 1\le i\le n.$$
This proves the assertions \eqref{4.1.2} and \eqref{4.1.3} of the theorem. 
\end{proof}

\section{Proof of Theorem \ref{key thm}}\label{V homomorphism section}

\subsection{Strategy of the proof}

In this section we always assume $\M$ to be a properly infinite $\sigma$-finite von Neumann algebra with a faithful normal semifinite trace $\tau$, and let $\varPhi=(E_1,\cdots,E_n)$ be an $n$-tuple of symmetric function spaces on $(0,\infty).$

We prove Theorem \ref{key thm} in several steps. The first step is to establish the existence of a smooth partition of the identity with good properties. The proof of the lemma below can be found in \cite[Theorem 4.3]{BSZZ2023}.

\begin{lemma}\label{smooth partition lem} Let $\alpha\in\mathcal{M}^n$ be a commuting self-adjoint $n$-tuple. Suppose $\varPhi=(E_1,\ldots,E_n)$ is an $n$-tuple of symmetric function spaces on $(0,\infty)$. Let $\varPhi(\M)=(E_1(\M),\ldots,E_n(\M))$. If $k_{\varPhi(\M)}(\alpha)=0,$ then for every $\varepsilon>0,$ there is a sequence $\{e_m\}_{m\ge 1}\subset\mathcal{F}_1^+(\M)$ such that 
$$\sum_{m\geq1}e_m^2 = \mathbf{1},\quad\sum_{m\geq1}\|[\alpha,e_m]\|_{\varPhi(\M)}\leq\varepsilon,$$
where the first series converges in the strong operator topology.
\end{lemma}

The second step for proving Theorem \ref{key thm} is the following assertion. Its proof is based on Lemma \ref{smooth partition lem} and Theorem \ref{isometries commutator vanishing thm}, and is given in the next subsection.

\begin{theorem}\label{pre thm 2.4}
Let $\alpha\in\M^n$ be a commuting self-adjoint $n$-tuple satisfying
$$WZ^*(\alpha)\cap \mathcal{P}_f(\M)=\{0\}.$$
Assume that $\psi\sim_{\M}{\rm Id}_{C^{\ast}(\alpha)}$ and that $k_{\varPhi(\M)}(\psi(\alpha))=0.$ There exists a sequence of isometries $\{v_l\}_{l\geq1}\subset\M$ such that
$$v_{l_1}^{\ast}v_{l_2}=\delta_{l_1,l_2}\mathbf{1},\quad l_1,l_2\in\mathbb{N},$$
$$v_l\psi(\alpha)-\alpha v_l\in \varPhi^0(\M),\quad \|v_l\psi(\alpha)-\alpha v_l\|_{\varPhi(\M)}\leq 2^{-l},\quad l\in\mathbb{N}.$$
\end{theorem}

The third step for proving Theorem \ref{key thm} is the following assertion. It is based on Theorem \ref{pre thm 2.4}. We denote by $\{{\bf E}_{ij}\}_{i,j=1}^2$ the matrix units in $M_2(\mathbb{C})$ (note that $\tr({\bf E}_{11})=1$).

\begin{theorem}\label{thm 2.4} Let $\alpha\in\M^n$ be a commuting self-adjoint $n$-tuple. Suppose 
$$WZ^{\ast}(\alpha)\cap \mathcal{P}_f(\M)=0.$$
Assume that $\psi\sim_{\M}{\rm Id}_{C^{\ast}(\alpha)}.$ Suppose
$$k_{\varPhi(\M)}(\alpha)=k_{\varPhi(\M)}(\psi(\alpha))=0.$$
For every $\e>0$, there exists an isometry $w\in\M\otimes M_2(\mathbb{C})$ such that $ww^{\ast}=\mathbf{1}\otimes {\bf E}_{11},$
$$w(\alpha\otimes {\bf E}_{11}+\psi(\alpha)\otimes {\bf E}_{22})-(\alpha\otimes {\bf E}_{11})w\in \varPhi^0(\M\otimes M_2(\mathbb{C})),$$
$$\|w(\alpha\otimes {\bf E}_{11}+\psi(\alpha)\otimes {\bf E}_{22})-(\alpha\otimes {\bf E}_{11})w\|_{\varPhi(\M\otimes M_2(\mathbb{C}))}\leq\varepsilon,$$
where $\varPhi(\M\otimes M_2(\mathbb{C}))=(E_1(\M\otimes M_2(\mathbb{C})),\ldots,E_n(\M\otimes M_2(\mathbb{C}))).$
\end{theorem}

Having Theorem \ref{thm 2.4} at hand, we prove Theorem \ref{key thm} in Subsection \ref{final subsection}.

\subsection{Proof of Theorem \ref{pre thm 2.4}}

\begin{proof}[Proof of Theorem \ref{pre thm 2.4}] Since $L_1\cap L_\infty$ is embedded continuously in $E_j$ for each $1\le j\le n$ (see Lemma \ref{intermidiate lem}), we may assume without of generality that 
$$\|x\|_{E_j}\le \|x\|_{L_1\cap L_\infty},\quad x\in L_1\cap L_\infty,$$
for each $1\le j\le n.$ 

From Lemma \ref{smooth partition lem}, for every $l\in \mathbb{N},$ there is a sequence $\{e_{m,l}\}_{m,l\geq 1}\subset \mathcal{F}_1^+(\M)$ such that $\sum_{m\geq1}e_{m,l}^2 = \mathbf{1}$ and
\begin{equation}\label{decom1} \sum_{m\geq1}\|[e_{m,l},\psi(\alpha)]\|_{\varPhi(\M)}\le 2^{-(l+1)}.
\end{equation}

As $\psi\sim_{\M}{\rm Id}_{C^{\ast}(\alpha)}$, we can find a sequence of unitaries $\{u_{m,l}\}_{m,l\geq1}\subset\M$ such that
\begin{equation}\label{decom2}
\|u_{m,l}\psi(\alpha)-\alpha u_{m,l}\|_{\M}\leq\frac1{2^{m+l+2}(1+\tau(e_{m,l}))},\quad m,l\in\mathbb{N}.
\end{equation}

As $\M$ is properly infinite, it follows from Theorem \ref{isometries commutator vanishing thm} that there exists a sequence of isometries $\{w_{m,l}\}_{m,l\geq1}\subset \M$ such that
$$w_{m_1,l_1}^\ast w_{m_2,l_2}=\delta_{m_1,m_2}\delta_{l_1,l_2}\mathbf{1},\quad m_1,m_2,l_1,l_2\in \mathbb{N},$$
\begin{equation}\label{decom3}
\|[w_{m,l},\alpha]\|_{\M}\leq\frac1{2^{m+l+2}(1+\tau(e_{m,l}))},\quad m,l\in\mathbb{N}.
\end{equation}

Set
$$v_l =\sum_{m\geq1} w_{m,l}u_{m,l}e_{m,l},\quad l\in\mathbb{N},$$
where the series converges in the strong operator topology. For $l_1,l_2\in \mathbb{N}$, we have
\begin{multline*}
v_{l_1}^{\ast}v_{l_2}=\sum_{m_1,m_2\geq 1}e_{m_1,l_1}u_{m_1,l_1}^{\ast}w_{m_1,l_1}^{\ast}w_{m_2,l_2}u_{m_2,l_2}e_{m_2,l_2}\\
=\sum_{m_1,m_2\geq 1}e_{m_1,l_1}u_{m_1,l_1}^{\ast}\cdot \delta_{m_1,m_2}\delta_{l_1,l_2}\mathbf{1}\cdot u_{m_2,l_2}e_{m_2,l_2}\\
=\delta_{l_1,l_2}\cdot \sum_{m\geq 1}e_{m,l_1}u_{m,l_1}^{\ast}\cdot u_{m,l_1}e_{m,l_1}=\delta_{l_1,l_2}\cdot \sum_{m\geq 1}e_{m,l_1}\cdot e_{m,l_1}=\delta_{l_1,l_2}\mathbf{1}.
\end{multline*}
In particular, each $v_l$ is an isometry.

Now for $1\leq j\leq n,$ we have
\begin{multline}\label{three summations eq}
v_l\psi(\alpha(j))-\alpha(j)v_l=\sum_{m\geq1}w_{m,l}u_{m,l}[e_{m,l},\psi(\alpha(j))]\\
+\sum_{m\geq1}w_{m,l}(u_{m,l}\psi(\alpha(j))-\alpha(j)u_{m,l})e_{m,l}+\sum_{m\geq1}[w_{m,l},\alpha(j)]u_{m,l}e_{m,l}.
\end{multline}

Since for $m,l\ge1$, $e_{m,l}\in \mathcal{F}(\M,\tau)$, it follows that every summand in the three summations in the right-hand side of \eqref{three summations eq} belongs to $\mathcal{F}(\M,\tau)$. We now show that the three series on the right-hand side of \eqref{three summations eq} also converge absolutely in $E_j(\M)$ and, therefore, the left-hand side of \eqref{three summations eq} belongs to $E_j^0(\M).$ By the triangle inequality, we have
\begin{multline*}
\|v_l\psi(\alpha(j))-\alpha(j)v_l\|_{E_j(\M)}\le\sum_{m\geq1}\|w_{m,l}u_{m,l}\|_{\M}\|[e_{m,l},\psi(\alpha(j))]\|_{E_j(\M)}\\
+\sum_{m\geq1}\|w_{m,l}\|_{\M}\|u_{m,l}\psi(\alpha(j))-\alpha(j)u_{m,l}\|_{\M}\|e_{m,l}\|_{E_j(\M)}\\
+\sum_{m\geq1}\|[w_{m,l},\alpha(j)]\|_{\M}\|u_{m,l}\|_{\M}\|e_{m,l}\|_{E_j(\M)}\\
\leq\sum_{m\geq1}\|[e_{m,l},\psi(\alpha(j))]\|_{E_j(\M)}+\sum_{m\geq1}\|u_{m,l}\psi(\alpha(j))-\alpha(j)u_{m,l}\|_{\M}\|e_{m,l}\|_{(L_1\cap L_{\infty})(\M)}\\
+\sum_{m\geq1}\|[w_{m,l},\alpha(j)]\|_{\M}\|e_{m,l}\|_{(L_1\cap L_{\infty})(\M)}
\end{multline*}
From \eqref{decom1}, \eqref{decom2} and \eqref{decom3}, we can conclude that
$$\|v_l\psi(\alpha(j))-\alpha(j)v_l\|_{E_j(\M)}\le 2^{-l}.$$
This completes the proof.
\end{proof}

\subsection{Proof of Theorem \ref{thm 2.4}}

\begin{lemma}\label{xk properties} Let $(v_l)_{l\geq1}\subset\M$ be isometries such that $v_{l_1}^{\ast}v_{l_2}=\delta_{l_1,l_2}\mathbf{1}$ for $l_1,l_2\ge1$:
\begin{enumerate}[{\rm (i)}]
\item\label{xpa} The series $\sum_{l\geq 1}v_lv_l^{\ast}$ and $\sum_{l\geq 1}v_{l+1}v_l^{\ast}$ converge in the strong operator topology.
\item\label{xpb} For $k\ge1,$ let $q_k=\sum_{l\ge k} v_lv_l^\ast$ and $t_k=\sum_{l\ge k}v_{l+1}v_l^\ast$. The operator
$$w_k=(\mathbf{1}-q_k+t_k)\otimes {\bf E}_{11}+v_k\otimes {\bf E}_{12}$$
is an isometry in $\M\otimes M_2(\mathbb{C})$ such that $w_kw_k^{\ast}=\mathbf{1}\otimes {\bf E}_{11}.$
\end{enumerate}
\end{lemma}
\begin{proof} For $l\ge1$, denote for brevity $p_l=v_lv_l^{\ast}$ and $u_l=v_{l+1}v_l^{\ast}.$ Clearly, $p_l$ is a projection and $u_l$ is a partial isometry such that
$$u_l^{\ast}u_l=v_lv_{l+1}^\ast\cdot v_{l+1}v_l^\ast=v_lv_l^\ast=p_l,\quad  u_lu_l^{\ast}=v_{l+1}v_l^\ast\cdot v_lv_{l+1}^\ast=v_{l+1}v_{l+1}^\ast=p_{l+1}.$$
Also, for $l_1,l_2\ge1$,
$$p_{l_1}p_{l_2}=v_{l_1}v_{l_1}^{\ast}\cdot v_{l_2}v_{l_2}^{\ast}=v_{l_1}\cdot \delta_{l_1,l_2}\mathbf{1}\cdot v_{l_2}^{\ast}=\delta_{l_1,l_2}v_{l_1}v_{l_1}^{\ast}=\delta_{l_1,l_2}p_{l_1}.$$
Hence, the series $\sum_{l\geq 1}p_l$ converges in the strong operator topology to a projection. Similarly, the series $\sum_{l\geq 1}u_l$ converges in the strong operator topology to a partial isometry. This proves the assertion \eqref{xpa}.

It is immediate that $q_k$ is a projection, and that $t_k$ is a partial isometry such that $t_k^{\ast}t_k=q_k$ and $t_kt_k^{\ast}=q_{k+1}.$ Thus, $\mathfrak{r}(t_k)=q_k$ and $\mathfrak{l}(t_k)=q_{k+1}.$ Hence, $(1-q_{k+1})t_k=0$ and $(1-q_k)t_k^{\ast}=0.$ Since $q_{k+1}\leq q_k,$ it follows that also $(1-q_k)t_k=0.$

For $k\ge1$, set $x_k=\mathbf{1}-q_k+t_k.$ Using the equalities from the paragraph above, we write
$$x_k^{\ast}x_k=(\mathbf{1}-q_k+t_k)^{\ast}(\mathbf{1}-q_k+t_k)=(\mathbf{1}-q_k)+t_k^{\ast}t_k=\mathbf{1},$$
$$x_kx_k^{\ast}=(\mathbf{1}-q_k+t_k)(\mathbf{1}-q_k+t_k)^{\ast}=(\mathbf{1}-q_k)+t_kt_k^{\ast}={\bf 1}-p_k.$$
Hence, $\mathfrak{l}(x_k)=\mathbf{1}-\mathfrak{r}(v_k^\ast)$ and thus $v_k^\ast x_k=0.$ Therefore,
\begin{multline*}
w_k^{\ast}w_k=(x_k^{\ast}\otimes {\bf E}_{11}+v_k^{\ast}\otimes {\bf E}_{21})(x_k\otimes {\bf E}_{11}+v_k\otimes {\bf E}_{12})\\
=x_k^{\ast}x_k\otimes {\bf E}_{11}+x_k^{\ast}v_k\otimes {\bf E}_{12}+v_k^{\ast}x_k\otimes {\bf E}_{21}+v_k^{\ast}v_k\otimes {\bf E}_{22}\\
=\mathbf{1}\otimes {\bf E}_{11}+0\otimes {\bf E}_{12}+0\otimes {\bf E}_{21}+\mathbf{1}\otimes {\bf E}_{22}=\mathbf{1}\otimes \mathbf{1}_{M_2(\mathbb{C})},
\end{multline*}
\begin{multline*}
w_kw_k^{\ast}=(x_k\otimes {\bf E}_{11}+v_k\otimes {\bf E}_{12})(x_k^{\ast}\otimes {\bf E}_{11}+v_k^{\ast}\otimes {\bf E}_{21})\\
=x_kx_k^{\ast}\otimes {\bf E}_{11}+v_kv_k^{\ast}\otimes {\bf E}_{11}=\mathbf{1}\otimes {\bf E}_{11},
\end{multline*}
where $\mathbf{1}_{M_2(\mathbb{C})}$ is the identity element in $M_2(\mathbb{C})$. This proves the assertion \eqref{xpb}.
\end{proof}

\begin{lemma}\label{24 main lemma} Let $(v_l)_{l\geq1}$ be the isometries given by Theorem \ref{pre thm 2.4}. Set
$$q_k=\sum_{l\ge k}v_lv_l^\ast,\quad t_k=\sum_{l\ge k}v_{l+1}v_l^\ast,\quad k\in\mathbb{N},$$
where the series converge in the strong operator topology. We have 
\begin{enumerate}[\rm (i)]
\item $[q_k,\alpha]\in \varPhi^0(\M)$ and $\|[q_k,\alpha]\|_{\varPhi(\M)}\le 2^{2-k};$
\item $[t_k,\alpha]\in\varPhi^0(\M)$ and $\|[t_k,\alpha]\|_{\varPhi(\M)}\le 2^{2-k}.$
\end{enumerate}
\end{lemma}
\begin{proof}
For $l\ge1,$
$$[v_{l}v_l^{\ast},\alpha]=v_{l}\cdot (\alpha v_l)^{\ast}-\alpha v_{l}\cdot v_l^{\ast}=v_{l}\cdot (\alpha v_l-v_l\psi(\alpha))^{\ast}-(\alpha v_{l}-v_{l}\psi(\alpha))v_l^{\ast}.$$
Thus,
$$[q_k,\alpha]=\sum_{l\geq k}[v_{l}v_l^{\ast},\alpha]=\sum_{l\geq k}v_{l}\cdot (\alpha v_l-v_l\psi(\alpha))^{\ast}-\sum_{l\geq k}(\alpha v_{l}-v_{l}\psi(\alpha))\cdot v_l^{\ast}.$$
By Theorem \ref{pre thm 2.4},
$$\alpha v_l-v_l\psi(\alpha)\in \varPhi^0(\M),\quad \|\alpha v_l-v_l\psi(\alpha)\|_{\varPhi(\M)}\le 2^{-l},\quad l\in\mathbb{N}.$$
By the triangle inequality, we have
$$\|[q_k,\alpha]\|_{\varPhi(\M)}
\leq 2\sum_{l\geq k}\|\alpha v_l-v_l\psi(\alpha)\|_{\varPhi(\M)}\le 2\cdot\sum_{l\ge k}2^{-l}= 2^{2-k}.$$
This proves the first assertion.

We now prove the second assertion. For $l\ge1$,
$$[v_{l+1}v_l^{\ast},\alpha]=v_{l+1}\cdot (\alpha v_l)^{\ast}-\alpha v_{l+1}\cdot v_l^{\ast}=v_{l+1}\cdot (\alpha v_l-v_l\psi(\alpha))^{\ast}-(\alpha v_{l+1}-v_{l+1}\psi(\alpha))v_l^{\ast}.$$
Thus,
$$[t_k,\alpha]=\sum_{l\geq k}[v_{l+1}v_l^{\ast},\alpha]=\sum_{l\geq k}v_{l+1}\cdot (\alpha v_l-v_l\psi(\alpha))^{\ast}-\sum_{l\geq k}(\alpha v_{l+1}-v_{l+1}\psi(\alpha ))\cdot v_l^{\ast}.$$
By the triangle inequality and Theorem \ref{pre thm 2.4},
\begin{multline*} \|[t_k,\alpha]\|_{\varPhi(\M)}\\
\leq \sum_{l\geq k}\|\alpha v_l-v_l\psi(\alpha)\|_{\varPhi(\M)}+\sum_{l\geq k}\|\alpha v_{l+1}-v_{l+1}\psi(\alpha)\|_{\varPhi(\M)}\le \sum_{l\ge k}2^{1-l}=2^{2-k}.
\end{multline*}
This proves the second assertion.
\end{proof}

\begin{proof}[Proof of Theorem \ref{thm 2.4}] Let $(v_l)_{l\geq1}$ be the isometries given by Theorem \ref{pre thm 2.4}. For $k\ge1$, let $q_k=\sum_{l\ge k}v_lv_l^\ast$ and $t_k=\sum_{l\ge k} v_{l+1}v_l^\ast,$ where the series converge in the strong operator topology by Lemma \ref{xk properties} \eqref{xpa}. Let $w_k\in\M\otimes M_2(\mathbb{C})$ be an isometry as in Lemma \ref{xk properties}. We have
\begin{multline*}
(\alpha\otimes {\bf E}_{11})w_k-w_k(\alpha\otimes {\bf E}_{11}+\psi(\alpha)\otimes {\bf E}_{22})\\
=(\alpha\otimes {\bf E}_{11})((\mathbf{1}-q_k+t_k)\otimes {\bf E}_{11}+v_k\otimes {\bf E}_{12})\\
-((\mathbf{1}-q_k+t_k)\otimes {\bf E}_{11}+v_k\otimes {\bf E}_{12})(\alpha\otimes {\bf E}_{11}+\psi(\alpha)\otimes {\bf E}_{22})\\
=\alpha(\mathbf{1}-q_k+t_k)\otimes {\bf E}_{11}+\alpha v_k\otimes {\bf E}_{12}\\
-(\mathbf{1}-q_k+t_k)\alpha\otimes {\bf E}_{11}-v_k\psi(\alpha)\otimes {\bf E}_{12}\\
=[q_k,\alpha]\otimes {\bf E}_{11}-[t_k,\alpha]\otimes {\bf E}_{11}+(\alpha v_k-v_k\psi(\alpha))\otimes {\bf E}_{12}.
\end{multline*}
By the triangle inequality, we have
\begin{multline*}
\|(\alpha\otimes {\bf E}_{11})w_k-w_k(\alpha\otimes {\bf E}_{11}+\psi(\alpha)\otimes {\bf E}_{22})\|_{\varPhi(\M\otimes M_2(\mathbb{C}))}\\
\leq\|[q_k,\alpha]\|_{\varPhi(\M)}+\|[t_k,\alpha]\|_{\varPhi(\M)}+\|\alpha v_k-v_k\psi(\alpha)\|_{\varPhi(\M)}.
\end{multline*}
Since $k\in\mathbb{N}$ is arbitrary, the assertion follows from Lemma \ref{24 main lemma} and Theorem \ref{pre thm 2.4}.
\end{proof}

\subsection{Proof of Theorem \ref{key thm}}\label{final subsection}

\begin{proof}[Proof of Theorem \ref{key thm}] Fix $\e>0.$ By Theorem \ref{thm 2.4}, there exists an isometry $w_1\in\M\otimes M_2(\mathbb{C})$ such that $w_1w_1^{\ast}=\mathbf{1}\otimes {\bf E}_{11},$
$$w_1(\alpha\otimes {\bf E}_{11}+\psi(\alpha)\otimes {\bf E}_{22})-(\alpha\otimes {\bf E}_{11})w_1\in \varPhi^0(\M\otimes M_2(\mathbb{C})),$$
\begin{equation}\label{main estimate 1}
\|w_1(\alpha\otimes {\bf E}_{11}+\psi(\alpha)\otimes {\bf E}_{22})-(\alpha\otimes {\bf E}_{11})w_1\|_{\varPhi(\M\otimes M_2(\mathbb{C}))}\leq \varepsilon.
\end{equation}

Since $\psi$ is injective and $\psi\sim_{\M}{\rm Id}_{C^{\ast}(\alpha)}$ by the assumption, it follows that $\psi^{-1}\sim_{\M} {\rm Id}_{C^{\ast}(\psi(\alpha))}$ on $C^\ast(\psi(\alpha))$. By assumption, $WZ^\ast(\psi(\alpha))\cap \mathcal{P}_f(\M)=\{0\}$. From Theorem \ref{thm 2.4} (applied to the $n$-tuple $\psi(\alpha)$ and the $\ast$-homomorphism $\psi^{-1}$), there exists an isometry $w_2\in\M \otimes M_2(\mathbb{C})$ such that $w_2w_2^{\ast}=\mathbf{1}\otimes {\bf E}_{11},$
$$w_2(\psi(\alpha)\otimes {\bf E}_{11}+\alpha\otimes {\bf E}_{22})-(\psi(\alpha)\otimes {\bf E}_{11})w_2\in \varPhi^0(\M\otimes M_2(\mathbb{C})),$$
$$\|w_2(\psi(\alpha)\otimes {\bf E}_{11}+\alpha\otimes {\bf E}_{22})-(\psi(\alpha)\otimes {\bf E}_{11})w_2\|_{\varPhi(\M\otimes M_2(\mathbb{C}))}\leq \varepsilon.$$
Taking adjoint, we obtain that
\begin{equation}\label{main estimate 2}
\|(\psi(\alpha)\otimes {\bf E}_{11}+\alpha\otimes {\bf E}_{22})w_2^\ast-w_2^\ast(\psi(\alpha)\otimes {\bf E}_{11})\|_{\varPhi(\M\otimes M_2(\mathbb{C}))}\leq \varepsilon.
\end{equation}
	
Set $y=w_1Uw_2^{\ast},$ where $U=\mathbf{1}\otimes ({\bf E}_{12}+{\bf E}_{21}).$ It is immediate that $y^{\ast}y=yy^{\ast}=\mathbf{1}\otimes {\bf E}_{11}.$ Hence, there exists a unitary element $u\in\mathcal{M}$ such that $y=u\otimes {\bf E}_{11}.$ We prove the assertion for exactly this $u.$
	
We have
\begin{multline*}
y(\psi(\alpha)\otimes {\bf E}_{11})-(\alpha\otimes {\bf E}_{11})y=w_1Uw_2^\ast(\psi(\alpha)\otimes {\bf E}_{11})-(\alpha\otimes {\bf E}_{11})w_1Uw_2^\ast\\
=w_1U(w_2^\ast(\psi(\alpha)\otimes {\bf E}_{11})-(\psi(\alpha)\otimes {\bf E}_{11}+\alpha\otimes {\bf E}_{22})w_2^\ast)\\
+w_1U(\psi(\alpha)\otimes {\bf E}_{11}+\alpha\otimes {\bf E}_{22})w_2^\ast-(\alpha\otimes {\bf E}_{11})w_1Uw_2^\ast\\
=w_1U(w_2^\ast(\psi(\alpha)\otimes {\bf E}_{11})-(\psi(\alpha)\otimes {\bf E}_{11}+\alpha\otimes {\bf E}_{22})w_2^\ast)\\
+(w_1(\alpha\otimes {\bf E}_{11}+\psi(\alpha)\otimes {\bf E}_{22})-(\alpha\otimes {\bf E}_{11})w_1)Uw_2^\ast,
\end{multline*}
where the last equality holds because
$$U(\psi(\alpha)\otimes {\bf E}_{11}+\alpha\otimes {\bf E}_{22})=(\alpha\otimes {\bf E}_{11}+\psi(\alpha)\otimes {\bf E}_{22})U.$$ 
By \eqref{main estimate 1}, \eqref{main estimate 2} and by the triangle inequality, we have 
$$y(\psi(\alpha)\otimes {\bf E}_{11})-(\alpha\otimes {\bf E}_{11})y\in \varPhi^0(\M\otimes M_2(\mathbb{C})),$$
\begin{multline*}
\|y(\psi(\alpha)\otimes {\bf E}_{11})-(\alpha\otimes {\bf E}_{11})y\|_{\varPhi(\M\otimes M_2(\mathbb{C}))}\\
\leq \|w_2^\ast(\psi(\alpha)\otimes {\bf E}_{11})-(\psi(\alpha)\otimes {\bf E}_{11}+\alpha\otimes {\bf E}_{22})w_2^\ast\|_{\varPhi(\M\otimes M_2(\mathbb{C}))}\\
+\|w_1(\alpha\otimes {\bf E}_{11}+\psi(\alpha)\otimes {\bf E}_{22})-(\alpha\otimes {\bf E}_{11})w_1\|_{\varPhi(\M\otimes M_2(\mathbb{C}))}\le2\varepsilon.
\end{multline*}
Clearly, 
$$y(\psi(\alpha)\otimes {\bf E}_{11})-(\alpha\otimes {\bf E}_{11})y=(u\psi(\alpha)-\alpha u)\otimes {\bf E}_{11}.$$
Thus, 
$$u\psi(\alpha)-\alpha u\in\varPhi^0(\M),$$
$$\|u\psi(\alpha)-\alpha u\|_{\varPhi(\M)}=\|(u\psi(\alpha)-\alpha u)\otimes {\bf E}_{11}\|_{\varPhi(\M\otimes M_2(\mathbb{C}))}\le 2\e.$$
Since $u\in\M$ is unitary, the assertion follows.
\end{proof}

\section{Proof of Theorem \ref{main theorem}}\label{final section}

We prove Theorem \ref{main theorem} in this section. We first focus on the hard part of Theorem \ref{main theorem}, namely, the implication \eqref{mta}$\to$\eqref{mtb}. Theorem \ref{key thm} and the construction in Section \ref{construction section} are two important ingredients in this process.

The case when $\M$ is finite is simple, and the proof is given below.

\begin{lemma}\label{approximation in finite case} Let $\M$ be a finite $\sigma$-finite von Neumann algebra with a faithful normal semifinite trace $\tau$ and let $n\in \mathbb{N}$. Let $\varPhi=(E_1,\ldots,E_n)$ be an $n$-tuple of symmetric function spaces on $(0,\infty).$ Suppose $\alpha\in\M^n$ is a commuting self-adjoint $n$-tuple. For every $\varepsilon>0,$ there exists a diagonal $n$-tuple $\delta\in\M^n$ such that $\alpha-\delta\in \varPhi^0(\M)$ and $\|\alpha-\delta\|_{\varPhi(\M)}\le\e.$
\end{lemma}
\begin{proof} Recall that $L_1\cap L_\infty$ is embedded continuously in any symmetric function space $E$ on $(0,\infty)$ (see Lemma \ref{intermidiate lem}). Without loss of generality, we may assume that $\|\cdot\|_{E_j}\leq\|\cdot\|_{L_1\cap L_{\infty}}$ for every $1\leq j\leq n.$

By Lemma \ref{finite and tau finite} (ii), there exists a central partition of the identity $\{z_i\}_{i\in I}$ in $\M$ such that $z_i$ is $\tau$-finite for each $i\in I.$ Since $\M$ is $\sigma$-finite, it follows that $I$ is countable, and thus we may write $I=\mathbb{N}.$

 By the spectral theorem, every commuting self-adjoint $n$-tuple in a von Neumann
algebra can be approximated in the uniform norm by diagonal $n$-tuples. For each $i\in \mathbb{N}$, choose a diagonal $n$-tuple $\delta_i\in (\M z_i)^n$ such that
\begin{equation}\label{estimate 1}
\|\alpha z_i-\delta_i\|_{\M}\le \frac{\e}{2^i(1+\tau(z_i))},\quad \|\delta_i\|_{\M}\le \|\alpha\|_{\M}.
\end{equation}

Let $\delta\in \M^n$ be defined by the formula
$$\delta(j)=\sum_{i=1}^\infty \delta_i(j),\quad 1\le j\le n.$$
Here, the series converges in the strong operator topology since $\{z_i\}_{i\in \mathbb{N}}$ are pairwise orthogonal and $\{\delta_i(j)\}_{i\in \mathbb{N}}$ are uniformly bounded for each $1\le j\le n$. 

For $1\le j\le n,$
$$\alpha(j)-\delta(j)=\sum_{i=1}^\infty (\alpha(j)z_i-\delta_i(j))$$
where the series converges absolutely in the uniform norm $\|\cdot\|_{\M}$ due to \eqref{estimate 1}.

For every $i\in \mathbb{N}$ and $1\le j\le n,$ we have 
\begin{multline*}
\|\alpha (j)z_i-\delta_i(j)\|_{(L_1\cap L_\infty)(\M)}\leq \|\alpha(j)z_i-\delta_i(j)\|_{L_1(\M)}+\|\alpha(j)z_i-\delta_i(j)\|_{\M}\\
\le \|\alpha(j) z_i-\delta_i(j)\|_{\M}\cdot (1+\tau(z_i))\leq\frac{\varepsilon}{2^i}.
\end{multline*}
Thus, we have
$$\|\alpha-\delta\|_{\varPhi(\M)}\le \sum_{i=1}^\infty\|\alpha z_i-\delta_i\|_{\varPhi(\M)}\leq \sum_{i=1}^\infty\|\alpha z_i-\delta_i\|_{(L_1\cap L_{\infty})(\M)}\leq\varepsilon.$$
For $i\ge1$, we have $\alpha z_i-\delta_i\in (\mathcal{F}(\M,\tau))^n$ since $\tau(z_i)<\infty$. Thus, $\alpha-\delta\in \varPhi^0(\M)$.
\end{proof}

In the remaining part, we deal with the case when $\M$ is properly infinite and semifinite. We need a $\ast$-monomorphism $\psi:C^\ast(\alpha)\to\M$ from Construction \ref{section2 main construction}. The following lemma shows that for such $\psi$, the $n$-tuple $\psi(\alpha)$ can be approximated by diagonal $n$-tuples in the sense of the $\|\cdot\|_{\varPhi(\M)}$-norm, where $\varPhi$ is any given $n$-tuple of symmetric function spaces on $(0,\infty)$.

\begin{lemma}\label{psialpha approximation lemma} Suppose we are in the setting of Construction \ref{section2 main construction}. Assume that $\M$ is equipped with a faithful normal semifinite trace $\tau$. Let $\varPhi=(E_1,\ldots,E_n)$ be an $n$-tuple of symmetric function spaces on $(0,\infty)$. For every $\varepsilon>0,$ there exists a diagonal $n$-tuple $\delta\in\mathcal{M}^n$ such that $\psi(\alpha)-\delta\in \varPhi^0(\M)$ and $\|\psi(\alpha)-\delta\|_{\varPhi(\mathcal{M})}\leq\varepsilon.$
\end{lemma}
\begin{proof} Without loss of generality, assume that $\|\chi_{(0,1)}\|_{E_j}=1$ for every $1\leq j\leq n.$

Let $\{p_k\}_{k\in \mathbb{N}}$ be the sequence of projections as in Construction \ref{section2 main construction}. For every $p_k,$ since $\M$ is $\sigma$-finite and semifinite, there is a sequence of non-zero pairwise orthogonal $\tau$-finite projections $\{q_{k,m}\}_{m\in \mathbb{N}}$ in $\M$ such that 
$$p_k = \sum_{m\in \mathbb{N}} q_{k,m}.$$
Without loss of generality, $\tau(q_{k,m})=1$ for every $k,m\ge1.$ Thus, $\|q_{k,m}\|_{E_j(\M)}=1$ for every $k,m\geq1$ and every $1\leq j\leq n.$ 

Fix $\e>0$ and fix $l\in\mathbb{N}$ such that $2^{-l}\le\varepsilon$. For $k,m\ge1$, let
$$f_{k,m}(x)=2^{-k-m-l}\lfloor 2^{k+m+l}x\rfloor,\quad x\in \mathbb{R}.$$
Here, $\lfloor x\rfloor$ denotes the integer part for $x\in \mathbb{R}$. 
Clearly, 
\begin{equation}\label{step function}
|f_{k,m}(x)-x|\le 2^{-k-m-l},\quad x\in \mathbb{R}.
\end{equation}
Set
$$\delta(j)=\sum_{k,m\in\mathbb{N}}f_{k,m}(\psi_k(\alpha(j))) q_{k,m},\quad 1\le j\le n.$$

For $k,m\ge1$, since the range of the function $f_{k,m}$ is countable, the operator $f_{k,m}(\psi_k(\alpha(j)))$ is diagonal for each $1\le j\le n.$ Since $\psi_k(\alpha(j))\in \mathcal{Z}(\M)$, it follows that 
$$f_{k,m}(\psi_k(\alpha(j)))\in \mathcal{Z}(\M),\quad 1\le j\le n.$$ 
Thus, $\delta$ is a commuting $n$-tuple. Note that $\{q_{k,m}\}_{k,m\in \mathbb{N}}$ are pairwise orthogonal. Hence, for every $1\le j\le n,$  $$\{f_{k,m}(\psi_k(\alpha(j)))q_{k,m}\}_{k,m\in \mathbb{N}}$$
is a sequence of pairwise orthogonal diagonal operators. Thus, $\delta(j)$ is diagonal for each $1\le j\le n$ and, finally, $\delta$ is a commuting diagonal $n$-tuple.

For $1\le j\le n,$
$$\psi(\alpha(j))\overset{{\rm Construction}\ \ref{section2 main construction}}{=}\sum_{k\in\mathbb{N}}\psi_k(\alpha(j))p_k=\sum_{k,m\in\mathbb{N}}\psi_k(\alpha(j))q_{k,m}.$$
Thus,
\begin{equation}\label{psi minus delta}
\psi(\alpha(j))-\delta(j)=\sum_{k,m\in\mathbb{N}} (\psi_k(\alpha(j))-f_{k,m}(\psi_k(\alpha(j)))) q_{k,m}.
\end{equation}
For $k,m\in \mathbb{N}$, $(\psi_k(\alpha(j))-f_{k,m}(\psi_k(\alpha(j)))) q_{k,m}\in \mathcal{F}(\M,\tau)$ since $\tau(q_{k,m})<\infty$. We now show that the series on the right-hand side of \eqref{psi minus delta} converges absolutely in $E_j(\M)$ and thus $\psi(\alpha(j))-\delta(j)\in E_j^0(\M)$. By the triangle inequality,
\begin{multline*}
\|\psi(\alpha(j))-\delta(j)\|_{E_j(\mathcal{M})}\le\sum_{k,m\in \mathbb{N}}\|(\psi_k(\alpha(j))-f_{k,m}(\psi_k(\alpha(j))))q_{k,m}\|_{E_j(\M)}\\
\le \sum_{k,m\in \mathbb{N}} \|\psi_k(\alpha(j))-f_{k,m}(\psi_k(\alpha(j)))\|_{\M}\|q_{k,m}\|_{E_j(\M)}\\
\overset{\eqref{step function}}{\le}\sum_{k,m\in\mathbb{N}}2^{-k-m-l}\|q_{k,m}\|_{E_j(\mathcal{M})}= \sum_{k,m\in\mathbb{N}}2^{-k-m-l}=2^{-l}\le\e.
\end{multline*}
Thus, $\psi(\alpha)-\delta\in \varPhi^0(\M)$ and $\|\psi(\alpha)-\delta\|_{\varPhi(\M)}\le \e$. 
\end{proof}

\begin{lemma}\label{verification lemma} Let $\M$ be a properly infinite $\sigma$-finite von Neumann algebra with a faithful normal semifinite trace $\tau$. Let $\alpha\in\M^n$ be a commuting self-adjoint $n$-tuple. Let $\varPhi=(E_1,\ldots,E_n)$ be an $n$-tuple of symmetric function spaces on $(0,\infty)$. Let a $C^\ast$-algebra $\mathcal{A}$ and $\psi\in {\rm Hom}(\mathcal{A},\M)$ be defined as in Construction \ref{section2 main construction}. Suppose also that $WZ^\ast(\alpha)\cap \mathcal{P}_f(\M)=\{0\}$ and $k_{\varPhi(\M)}(\alpha)=0.$ The triple $(\M,\alpha,\psi)$ satisfies the assumptions in Theorem \ref{key thm}.
\end{lemma}
\begin{proof} Recall that in Construction \ref{section2 main construction}, we let $\{p_k\}_{k\ge1}$ be a sequence of pairwise orthogonal projections such that $p_k\sim {\bf 1}$ for every $k\ge1$, let $\{\psi_k\}_{k\ge1}$ be a separating family in ${\rm Hom}(\mathcal{A},\mathcal{Z}(\M))$, and that $\psi$ is defined by setting
\begin{equation}\label{construction psi}
\psi(a)=\sum_{k\ge1}\psi_k(a)p_k,\quad a\in \mathcal{A}.
\end{equation}
By Theorem \ref{approximation by inner hom} \eqref{psi point 2}, $\psi$ is a faithful $\ast$-homomorphism on $\mathcal{A}$. This implies that $\psi$ is an isometry on $\mathcal{A}$ (see e.g. \cite[Corollary I.5.4]{T1}) and so $\psi(C^\ast(\alpha))=C^\ast(\psi(\alpha)).$ 

From Theorem \ref{approximation by inner hom} \eqref{psi point 3}, we have
$$\psi|_{C^\ast(\alpha)}\sim_{\M}{\rm Id}_{C^{\ast}(\alpha)}.$$
For $a\in\alpha$, since $\psi_k(a)\in \mathcal{Z}(\M)$ for each $k\ge1$, it follows from \eqref{construction psi} that $\psi(a)\in WZ^\ast(\{p_k\}_{k\in \mathbb{N}}).$ Hence, $WZ^\ast(\psi(\alpha))\subset WZ^\ast(\{p_k\}_{k\in \mathbb{N}})$. Since each $p_k\sim \mathbf{1},$ $k\ge1,$ is properly infinite, it follows that
$$WZ^\ast(\psi(\alpha))\cap \mathcal{P}_f(\M)\subset WZ^\ast(\{p_k\}_{k\in \mathbb{N}})\cap \mathcal{P}_f(\M)\stackrel{L.\ref{properly infinite sequence lemma}}{=}\{0\}.$$

By Lemma \ref{psialpha approximation lemma}, there exists a diagonal $n$-tuple $\delta\in \M^n$ such that $\psi(\alpha)-\delta\in\varPhi^0(\M)$. This implies that $k_{\varPhi(\M)}(\psi(\alpha))=k_{\varPhi(\M)}(\delta)$ (see \cite[Proposition 3.1]{BSZZ2023}). Since $\delta\in\M^n$ is a diagonal $n$-tuple, it follows from Lemma \ref{diagonal has zero k} that $k_{\varPhi(\M)}(\delta)=0.$ Hence, $k_{\varPhi(\M)}(\psi(\alpha))=0.$ This completes the proof.
\end{proof}

\begin{lemma}\label{extreme non-abelian case}  Let $\M$ be a properly infinite $\sigma$-finite von Neumann algebra with a faithful normal semifinite trace $\tau$. Suppose $\varPhi=(E_1,\ldots,E_n)$ is an $n$-tuple of symmetric function spaces on $(0,\infty)$.  Let $\alpha\in\M^n$ be a commuting self-adjoint $n$-tuple. Suppose also that $WZ^\ast(\alpha)\cap \mathcal{P}_f(\M)=\{0\}$ and $k_{\varPhi(\M)}(\alpha)=0.$ Let a $C^\ast$-algebra $\mathcal{A}$ and $\psi\in {\rm Hom}(\mathcal{A},\M)$ be defined as in Construction \ref{section2 main construction}. For every $\e>0,$ there exists a unitary $u\in\M$ such that $u\psi(\alpha)u^{-1}-\alpha\in \varPhi^0(\M)$ and
$$\|u\psi(\alpha)u^{-1}-\alpha\|_{\varPhi(\M)}<\varepsilon.$$
\end{lemma}
\begin{proof} The assertion follows now from Theorem \ref{key thm} and Lemma \ref{verification lemma}.
\end{proof}

Suppose $\mathcal{N}$ is a von Neumann subalgebra of $\M$ and suppose $\beta\subset \mathcal{N}$ is a subset. We denote by $WZ_{\mathcal{N}}^\ast(\beta)$ the von Neumann subalgebra in $\mathcal{N}$ generated by $\mathcal{Z}(\mathcal{N})\cup\beta$.

\begin{lemma}\label{fedor request lemma} Let $\M$ be a properly infinite von Neumann algebra with a faithful normal semifinite trace $\tau$ and let $\alpha\in\M^n$ be a commuting self-adjoint $n$-tuple. Set
$$p=\bigvee\{x:x\in WZ^\ast(\alpha)\cap \mathcal{P}_f(\M)\},$$
$$\mathcal{W}=pWZ^\ast(\alpha),\quad \mathcal{N}=(\mathbf{1}-p)\M(\mathbf{1}-p),\quad \beta=(\mathbf{1}-p)\alpha.$$
We have
\begin{enumerate}[{\rm (i)}]
\item $p$ commutes with $\alpha;$
\item $\tau|_{\mathcal{W}}$ is semifinite;
\item $\mathcal{N}$ is properly infinite;
\item $WZ_{\mathcal{N}}^\ast(\beta)\cap \mathcal{P}_f(\mathcal{N})=\{0\}.$
\end{enumerate}
\end{lemma}
\begin{proof} The first assertion is obvious. The second assertion follows immediately from Lemma \ref{finite and tau finite} \eqref{ftfa}.

Suppose $z\in \mathcal{Z}(\M)$ is a projection such that $(\mathbf{1}-p)z$ is finite and non-zero. Clearly, $(\mathbf{1}-p),z\in WZ^\ast(\alpha).$ Thus, $(\mathbf{1}-p)z\in WZ^\ast(\alpha).$ Since $(\mathbf{1}-p)z$ is finite, then, by the definition of $p,$ we have $(\mathbf{1}-p)z\leq p.$ However, $(\mathbf{1}-p)z\leq \mathbf{1}-p.$ Thus, $(\mathbf{1}-p)z\leq p\wedge (\mathbf{1}-p)=0.$

By the preceding paragraph, for every projection $z\in \mathcal{Z}(\M),$ the projection $(\mathbf{1}-p)z$ is either zero or infinite. Hence, the projection $\mathbf{1}-p$ is properly infinite. In other words, $\mathcal{N}$ is properly infinite, which proves the third assertion.

It is immediate that
$$WZ_{\mathcal{N}}^\ast(\beta)=(\mathbf{1}-p)WZ^{\ast}(\alpha)=\{x\in WZ^{\ast}(\alpha):\ px=0\}.$$
Thus,
\begin{multline*}
WZ_{\mathcal{N}}^\ast(\beta)\cap \mathcal{P}_f(\mathcal{N},\tau)\\
=\{x\in WZ^{\ast}(\alpha):\ px=0\}\cap \{x\in \mathcal{P}_f(\mathcal{M}):\ px=xp=0\}\\
=\{x\in WZ^{\ast}(\alpha)\cap \mathcal{P}_f(\mathcal{M}):\ px=0\}.
\end{multline*}
However, if $x\in WZ^{\ast}(\alpha)\cap \mathcal{P}_f(\mathcal{M}),$ then, by the definition of $p,$ we have $x=px.$ If also $px=0,$ then inevitably $x=0.$ This proves the fourth assertion.
\end{proof}

The following lemma yields the implication \eqref{mta}$\to$\eqref{mtb} of Theorem \ref{main theorem} provided that $\M$ is properly infinite.

\begin{lemma}\label{pre-final lemma} Suppose we are in the setting of Theorem \ref{main theorem}. Suppose also that $\M$ is properly infinite. If $k_{\varPhi(\M)}(\alpha)=0,$ then, for every $\varepsilon>0,$ there exists a diagonal $n$-tuple $\delta\in\M^n$ such that
$$\alpha-\delta\in\varPhi^0(\M),\quad \|\alpha-\delta\|_{\varPhi(\M)}<\varepsilon.$$
\end{lemma}
\begin{proof} Let $p,$ $\mathcal{W},$ $\mathcal{N}$ and $\beta$ be as in Lemma \ref{fedor request lemma}.

Applying Lemma \ref{approximation in finite case} to the commutative von Neumann algebra $\mathcal{W},$ to the semifinite trace $\tau|_{\mathcal{W}}$ and to the $n$-tuple $p\alpha,$ we find a diagonal $n$-tuple $\delta_1\in \mathcal{W}^n$ such that
$$p\alpha-\delta_1\in\varPhi^0(\mathcal{W}),\quad \|p\alpha-\delta_1\|_{\varPhi(\mathcal{W})}<\varepsilon.$$

Since $k_{\varPhi(\M)}(\alpha)=0$, by Lemma \ref{qm vanishes on subspace} we have
$$k_{\varPhi(\mathcal{N})}(\beta)=0.$$
By Lemma \ref{fedor request lemma}, $\mathcal{N}$ is properly infinite and  $WZ_{\mathcal{N}}^\ast(\beta)\cap \mathcal{P}_f(\mathcal{N})=\{0\}.$ Let $\psi\in {\rm Hom}(C^{\ast}(\beta),\mathcal{N})$ be defined as in Construction \ref{section2 main construction}. By Lemma \ref{extreme non-abelian case}, there exists a unitary $u\in\mathcal{N}$ such that
$$u\psi(\beta)u^{-1}-\beta\in\varPhi^0(\mathcal{N}),\quad \|u\psi(\beta)u^{-1}-\beta\|_{\varPhi(\mathcal{N})}<\varepsilon.$$
By Lemma \ref{psialpha approximation lemma}, there exists a diagonal $n$-tuple $\delta_2\in\mathcal{N}^n$ such that
$$\psi(\beta)-\delta_2\in\varPhi^0(\mathcal{N}),\quad \|\psi(\beta)-\delta_2\|_{\varPhi(\mathcal{N})}<\e.$$
Set $\delta_3=u\delta_2u^{-1}\in\mathcal{N}^n.$ Since $u$ is unitary, it follows that  $\delta_3$ is diagonal and 
$$u\psi(\beta)u^{-1}-\delta_3\in \varPhi^0(\mathcal{N}),\quad \|u\psi(\beta)u^{-1}-\delta_3\|_{\varPhi(\mathcal{N})}<\e.$$
By the triangle inequality,
$$\beta-\delta_3\in\varPhi^0(\mathcal{N}),\quad \|\beta-\delta_3\|_{\varPhi(\mathcal{N})}<2\varepsilon.$$

Set now $\delta=\delta_1+\delta_3.$ Since $\delta_1\in\mathcal{W}^n$ and $\delta_3\in\mathcal{N}^n$ are mutually orthogonal diagonal $n$-tuples, it follows that $\delta\in\M^n$ is also a diagonal $n$-tuple. By the triangle inequality, we have $\alpha-\delta\in \varPhi^0(\M)$ and
\begin{multline*}
\|\alpha-\delta\|_{\varPhi(\M)}\leq\|p\alpha-\delta_1\|_{\varPhi(\M)}+\|\beta-\delta_3\|_{\varPhi(\M)}\\
=\|p\alpha-\delta_1\|_{\varPhi(\mathcal{W})}+\|\beta-\delta_3\|_{\varPhi(\mathcal{N})}<3\varepsilon.
\end{multline*}
This completes the proof.
\end{proof}

Finally, we are in the position to give the proof of Theorem \ref{main theorem}.

\begin{proof}[Proof of Theorem \ref{main theorem}] The implication \eqref{mtb}$\to$\eqref{mta} is easy. Indeed, since $\alpha-\delta\in\varPhi^0(\M)$, it follows from \cite[Proposition 3.1]{BSZZ2023} that $k_{\varPhi(\M)}(\alpha)=k_{\varPhi(\M)}(\delta)$. By Lemma \ref{diagonal has zero k}, $k_{\varPhi(\M)}(\delta)=0.$ Hence, $k_{\varPhi(\M)}(\alpha)=0.$

We now prove the implication \eqref{mta}$\to$\eqref{mtb}. Fix $\alpha\in\M^n$ and assume $k_{\varPhi(\M)}(\alpha)=0.$ 

There exists a central projection $z\in \M$ such that $z\M$ is finite and $(\mathbf{1}-z)\M$ is properly infinite (see e.g. \cite[Proposition V.1.19]{T1}).

Fix $\e>0.$ Applying Lemma \ref{approximation in finite case} to the pair $(z\M,z\alpha)$, we find a diagonal $n$-tuple $\delta_1\in (z\M)^n$ such that $z\alpha-\delta_1\in \varPhi^0(z\M)$ and 
\begin{equation}\label{approximate finite part}
\|z\alpha-\delta_1\|_{\varPhi(z\M)}\le \e.
\end{equation} 

Since $(\mathbf{1}-z)$ commutes with $\alpha$, it follows from Lemma \ref{qm vanishes on subspace} that
$$k_{\varPhi((\mathbf{1}-z)\M)}((\mathbf{1}-z)\alpha)=0.$$
Applying Lemma \ref{pre-final lemma} to the pair $(({\bf 1}-z)\M,({\bf 1}-z)\alpha)$, we find a diagonal $n$-tuple $\delta_2\in (({\bf 1}-z)\M)^n$ such that $({\bf 1}-z)\alpha-\delta_2\in \varPhi^0(({\bf 1}-z)\M)$ and 
\begin{equation}\label{approximate infinite part}
\|({\bf 1}-z)\alpha-\delta_2\|_{\varPhi(({\bf 1}-z)\M)}\le \e.
\end{equation}

Set $\delta=\delta_1+\delta_2$. Since $\delta_1\in ( z\M)^n, \delta_2\in (({\bf 1}-z)\M)^n$ are mutually orthogonal diagonal $n$-tuples, it follows that $\delta\in\M^n$ is a diagonal $n$-tuple. Combining \eqref{approximate finite part} with \eqref{approximate infinite part} and using the triangle inequality, we obtain that $\alpha-\delta\in \varPhi^0(\M)$ and $$\|\alpha-\delta\|_{\varPhi(\M)}\le \|z\alpha-\delta_1\|_{\varPhi(\M)}+\|({\bf 1}-z)\alpha-\delta_2\|_{\varPhi(\M)}\le 2\e.$$
This proves the implication \eqref{mta}$\to$\eqref{mtb}.
\end{proof}

\section*{Acknowledgement}
F. Sukochev and D. Zanin are supported by the Australian Research Council DP230100434. H. Zhao acknowledges the support of Australian Government Research Training Program (RTP) Scholarship.

\end{document}